\documentclass{amsart}
\usepackage{amssymb,latexsym,amsmath,amsfonts}
\newtheorem{definition}{Definition}[section]
\newtheorem{thm}[definition]{Theorem}
\newtheorem{cor}[definition]{Corollary}
\newtheorem{lem}[definition]{Lemma}
\newtheorem{prop}[definition]{Proposition}

\newtheorem{eg}{Example}[section]
\newtheorem{rem}{Remark}[section]

\newcommand{\F}{\mathbb{F}_q}
\newcommand{\Fn}{\mathbb{F}_{q^n}}

\newcommand{\ord}{\operatorname{ord}}
\newcommand{\lcm}{\operatorname{lcm}}
\newcommand{\tab}{\hspace*{2em}}
\newcommand{\G}{\Gamma_q}

\begin{document}

\title{Composed Products and Explicit Factors of Cyclotomic Polynomials
over Finite Fields}
\author{Aleksandr Tuxanidy}
\address{School of Mathematics and Statistics, Carleton
University, 1125 Colonel By Drive, Ottawa, Ontario, K1S 5B6,
Canada.} 
\email{aleksandrtt@hotmail.com}

\author{Qiang Wang}
\address{School of Mathematics and Statistics, Carleton
University, 1125 Colonel By Drive, Ottawa, Ontario, K1S 5B6,
Canada.} 
\email{wang@math.carleton.ca} 

\thanks{Aleksandr Tuxanidy wishes to dedicate his work here to
Dr. E. Lorin and Dr. Q. Wang for their support and guidance throughout the years
2010, 2011. In particular, they made him believe in himself as a student once more. The research of
 Qiang Wang is partially supported by NSERC of Canada.}


\maketitle

\begin{abstract}
Let $q = p^s$ be a power of a prime number $p$ and let $\F$ be the finite field with $q$ elements. 
In this paper we obtain the
explicit factorization of the cyclotomic polynomial $\Phi_{2^nr}$ over $\F$
where both $r \geq 3$ and $q$ are odd, $\gcd(q,r) = 1,$ and $n\in \mathbb{N}.$
Previously, only the special cases when $r = 1,\ 3,\ 5,$ had been achieved. 
For this we make the assumption that the explicit factorization of $\Phi_r$
over $\F$ is given to us as a known. Let $n = p_1^{e_1}p_2^{e_2} \cdots p_s^{e_s}$
be the factorization of $n \in \mathbb{N}$ into powers of distinct primes $p_i,\ 1\leq i \leq s.$ In the case that the orders of $q$ modulo all these prime powers $p_i^{e_i}$ are pairwise coprime we show how to obtain the explicit factors of $\Phi_{n}$ from
the factors of each $\Phi_{p_i^{e_i}}.$ We also demonstrate 
how to obtain the factorization of $\Phi_{mn}$ from the factorization of
$\Phi_n$ when $q$ is a primitive root modulo $m$ and $\gcd(m,n) =
\gcd(\phi(m),\ord_n(q)) = 1.$ Here $\phi$ is the Euler's totient function, and
$\ord_n(q)$ denotes the multiplicative order of $q$ modulo $n.$ Moreover, we present the construction
of a new class of irreducible polynomials over $\F$ and generalize a result due to Varshamov
(1984) \cite{Varshamov}.
\end{abstract}

\section{Introduction}

\subsection{Composed Products and Applications}
Let $q = p^s$ be a power of a prime $p,$ and $\F$ be a finite field with $q$
elements. The multiplicative version of composed products of
two polynomials $f,\ g \in \F[x]$ (or composed multiplication for
short) defined by $$(f \odot g)(x) = \prod_{\alpha}\prod_{\beta} (x - \alpha
\beta)$$ where the product $\prod_\alpha \prod_{\beta}$ runs over all roots $\alpha,\
\beta$ of $f,\ g$ respectively, was first introduced by Selmer (1966)
\cite{Selmer} for the purposes of studying linear recurrence sequences (LRS).
Informally, LRS's are sequences whose terms depend linearly on a finite number of its predecessors;
thus a famous example of a LRS is the Fibonacci sequence whose terms are the
sum of the previous two terms. Let $k$ be a positive integer and
let $a,a_0,\dots,a_{k-1}$ be given elements in $\F.$ Then a sequence $S =
\{s_0,s_1,\dots\}$ of elements $s_i \in \F$ satisfying the relation
$$s_{n+k} = a_{k-1}s_{n+k-1} + a_{k-2}s_{n+k -2} + \dots + a_0s_n + a,\tab
n=0,1,\dots$$
is a LRS. If $a = 0,$ then $S$ is called a \emph{homogeneous} LRS. If we let $k
= 2,\ a = 0,\ a_0 = a_1 = 1$ and $s_0 = 0,\ s_1 = 1$ then $S$ becomes the
(homogeneous) Fibonacci sequence. LRS's have applications in coding theory,
cryptography, and other areas of electrical engineering where electric switching circuits such as linear feedback shift registers (LFSR) 
are used to generate them. See Chapter 8 in \cite{Lidl} for this and a general
introduction. In particular, the matter of the linear complexity of a LRS,
and more generally, the linear complexity of the component wise
multiplication of LRS's, is of great importance in stream cipher theory, a branch in cryptography;
here a higher complexity is preferred. See \cite{Gao} for instance and the
references contained therein. Since the linear complexity of a LRS is given by
the degree of the minimal polynomial of the LRS, minimal polynomials with
higher degrees are therefore preferred. 

The polynomial
$$f(x) = x^k -a_{k-1}x^{k-1} - a_{k-2}x^{k-2} - \dots - a \in \F[x]$$
is called the \emph{characteristic polynomial of S} (see \cite{Lidl}). In 1973,
Zierler and Mills \cite{Zierler} showed that the characteristic polynomial of a component wise multiplication of homogeneous LRS's is the composed
multiplication of the characteristic polynomials of the respective LRS's. That
is, if $S_1,S_2,\dots,S_r$ are homogeneous LRS's with respective characteristic
polynomials $f_1,f_2,\dots, f_r,$ then the characteristic polynomial
of $S_1S_2 \cdots S_r,$ with component wise multiplication, is given by $f_1
\odot f_2 \odot \dots \odot f_r.$ We refer the reader to page 433-435 in
\cite{Lidl} as well. Note that since the required minimal polynomials are factors of the characteristic polynomials $f_1 \odot f_2 \odot \dots \odot f_r$ of LRS's,
the study of factorizations of composed products has an important
significance. Thus composed products have applications in stream cipher theory,
LFSR, and LRS in general.

Similarly, the \emph{composed
sum} of $f, g \in \F[x]$ is defined by 
$$(f \oplus g)(x) = \prod_\alpha \prod_\beta (x - (\alpha + \beta)) $$
where the product runs over all the roots $\alpha$ of $f$ and
$\beta$ of $g,$ including multiplicities. 

In 1987, Brawley and Carlitz \cite{Brawley and Carlitz} generalized composed
multiplications and composed sums  in the following.

\begin{definition}{\bf \cite{Brawley and Carlitz} (Composed
Product)} 
Let
$G$ be a non-empty subset of the algebraic closure $\Gamma_q$ of $\F$ with the
property that $G$ is invariant under the Frobenius automorphism $\alpha
\mapsto \sigma(\alpha) = \alpha^q$ (i.e., if $\alpha \in G,$ then
$\sigma(\alpha) \in G$). Suppose a binary operation
$\diamond$ is defined on $G$ satisfying
$\sigma(\alpha \diamond \beta) = \sigma(\alpha)\diamond \sigma(\beta)$ for all
$\alpha,\beta \in G.$ Then the \emph{composed product} of $f$ and $g,$ denoted
by $f \diamond g,$ is the polynomial defined by
$$(f \diamond g)(x) = \prod_\alpha \prod_\beta (x - (\alpha \diamond \beta)), $$
where the $\diamond$-products run over all roots $\alpha$ of $f$ and $\beta$
of $g.$
\end{definition}

Observe that $\deg (f \diamond g) = (\deg
f)(\deg g)$ clearly. Moreover, in \cite{Brawley and Carlitz} it is 
noted that when $G = \G -\{0\}$ (respectively, $\G$) and $\diamond$ is the usual multiplication
(respectively, addition) then $f \diamond g$ becomes $f \odot g$ (respectively,
$f \oplus g,$). Other less common examples are

(i) $G = \G,\ \alpha \diamond \beta = \alpha + \beta - c$ where $c \in
\F$ is fixed.

(ii) $G = \G - \{1\},\ \alpha \diamond \beta = \alpha + \beta - \alpha\beta$ (sometimes called
the circle product), and

(iii) $G =$ any $\sigma$-invariant subset of $\G, \alpha \diamond \beta =
f(\alpha,\beta)$ where $f(x,y)$ is any fixed polynomial in $\F[x,y]$ such that
$f(\alpha,\beta) \in G$ for all $\alpha, \beta \in G.$

Let $M_G[q,x]$ be the set of
all monic polynomials over $\F$ of degree $\geq 1$ whose roots lie in $G.$
It is also shown in \cite{Brawley and Carlitz} that the condition $\sigma(\alpha
\diamond \beta) = \sigma(\alpha)\diamond \sigma(\beta)$ implies that $f
\diamond g \in \F[x].$ Moreover, if $\diamond$ is an associative
(respectively, commutative) product on $G,$ the composed product is associative (respectively, commutative) on
$M_G[q,x].$ In particular, composed multiplications and sums of
polynomials are associative and commutative in $\F[x].$  In fact, $(G, \diamond)$ is an abelian
group for composed multiplication, composed addition, and the example in (i) or (ii).

\subsection{Irreducible Constructions}
The construction of irreducible polynomials over finite fields is currently a
strong subject of interest with important applications in coding theory,
cryptography and complexity theory (\cite{Cohen}, \cite{Cohen 2005},
\cite{Kyuregyan}, \cite{Lidl}, \cite{Varshamov}). One of the most popular
methods of construction is the method of composition of polynomials (not to be
confused with composed products) where an irreducible polynomial of a higher
degree is produced from a given irreducible polynomial of lower degree by
applying a substitution operator. For a recent survey of previous results up to
the year 2005 on this subject see \cite{Cohen 2005}. Perhaps one of the most
applicable results in this area is the following.

\begin{thm}[{\bf Cohen (1969)}]\label{cohen thm}
Let $f$ and $g$ be two non-zero relatively prime irreducible polynomials
over $\F$ and $P$ be an irreducible polynomial over $\F$ of degree $n >0.$
Then the composition
$$
F(x) = g(x)^nP\left(f(x)/g(x)\right)
$$
is irreducible over $\F$ if and only if $f - \alpha g$ is irreducible over
$\mathbb{F}_{q^n}$ for some root $\alpha \in \mathbb{F}_{q^n}$ of $P.$
\end{thm} 

Note that Theorem \ref{cohen thm} has been used extensively in the past by
several authors in order to produce iterative constructions of irreducible
polynomials. See \cite{Cohen 2005} for instance and the references there.

Recently, Kyuregyan-Kyureghyan provides another proof of
Theorem \ref{cohen thm} in \cite{Kyuregyan} using the idea of composing factors of irreducible
polynomials over extension fields. Suppose $f$ is an irreducible polynomial
over $\F$ of degree $n$ and $g(x) = \sum_{i=0}^{n/d} g_i x^i \in
\mathbb{F}_{q^d}[x]$ is a factor of $f.$ Then all the remaining factors are
\[
g^{(u)}(x) =  \sum_{i=0}^{n/d} g_i^{q^u} x^i,
\]
where $1\leq u \leq d-1.$ We denote  $g = g^{(0)},$ and thus $f =
\prod_{u=0}^{d-1} g^{(u)}.$ Conversely, given an irreducible polynomial
$g$ of degree $n/d$ over $\mathbb{F}_{q^d},$ we can form the product $f
= \prod_{u=0}^{d-1} g^{(u)}.$ However, $f$ is not always an irreducible
polynomial over $\F.$ It is an irreducible polynomial only when
$\mathbb{F}_{q^d}$ is the smallest extension field of $\F$ containing the
coefficients of $g,$ i.e., when $\F(g_0, \dots, g_k) = \mathbb{F}_{q^d}.$ In
particular, they obtain the following.

\begin{thm}[\bf Theorem 1, \cite{Kyuregyan}]\label{kk thm}
Let $k > 1$, $\gcd(k, d) =1$, and  $f$ be an irreducible polynomial of degree $k$ over $\F$. Further let $\alpha \neq 0$ and $\beta$ be elements of $\mathbb{F}_{q^d}$. Set $g(x) := f(\alpha x + \beta)$. Then the polynomial
\[
F = \prod_{u=0}^{d-1} g^{(u)}
\]
of degree $n = dk$ is irreducible over $\F$ if and only if $\F(\alpha, \beta) = \mathbb{F}_{q^d}$.
\end{thm} 

We note that besides the above results there are 
others that are, perhaps, equally applicable in this area. In
particular, a result due to Brawley and Carlitz (1987) \cite{Brawley
and Carlitz}, is also instrumental in the construction of
irreducible polynomials of relatively higher degree from given polynomials of
relatively lower degrees. 

\begin{thm}[{\bf Theorem 2, \cite{Brawley and Carlitz}}]\label{thm 1x} 
Suppose
that $(G,\diamond)$ is a group and let $f,g \in M_G[q,x]$ with $\deg f = m$ and
$\deg g = n.$ Then the composed product $f \diamond g$ is irreducible if and
only if $f$ and $g$ are both irreducible with $\gcd(m,n) = 1.$
\end{thm}

In Section 2 we construct irreducible polynomials through the use of composed
products. First, we show that for some choices of $\alpha,\ \beta,$ the product of irreducible polynomials in Theorem~\ref{kk thm},
$F,$ is in fact a composed product, and therefore can be derived from 
Theorem~\ref{thm 1x}. Moreover, we obtain several concrete constructions of
irreducible polynomials (Theorem~\ref{thm 2} and Theorem~\ref{thm 3}) where
Theorem~\ref{thm 3} generalizes a classical result due to Varshamov
\cite{Varshamov} (see also Theorem 3 \cite{Kyuregyan}) and both
Theorems~\ref{thm 2}, \ref{thm 3}, use cyclotomic polynomials as one of two
inputs of composed products.

\subsection{Factorization of Cyclotomic Polynomials}
Let $\Phi_n$ denote the $n$-th cyclotomic polynomial
$$
\Phi_n(x) = \prod_{0 < j \leq n,\ (j,n) = 1}\left(x-\xi_n^j\right)
$$
where $\xi_n$ is a primitive $n$-th root of unity. Clearly, $x^n-1 = \prod_{d
\mid n}\Phi_{d}(x)$ and the Mobius Inversion Formula gives $\Phi_n(x) =
\prod_{d \mid n}(x^d-1)^{\mu(n/d)}$ where $\mu$ is the Mobius function.
Cyclotomic polynomials have been studied extensively since they first appeared 
in the 18th century works of Euler, Lagrange, Gauss, and others, and to
this day continue to be a strong subject of interest in Mathematics
(\cite{Bamunoba}, \cite{Sury}, \cite{Washington}).
This is a class of polynomials which naturally
arise in the classical 2000 year old Greek problem of Cyclotomy which concerns
the division of the circumference of the unit circle into $n$ equal parts, a
problem that was finally solved by Gauss at the turn of the 19th century.

It is well known the fact that all cyclotomic polynomials are irreducible over
the field of rational numbers; this is not the case over finite fields. In fact,
$\Phi_n$ decomposes into $\phi(n)/d$ irreducibles over $\F$ of the same
degree $d = \ord_n(q)$ (see Theorem 2.47 in \cite{Lidl}). The first steps in the
factorization of cyclotomic polynomials over finite fields were made in the 19th
century by Gauss, Pellet and others who restricted 
their studies to the prime fields $\mathbb{F}_p$ (p. 77, \cite{Lidl}). More
recently, Fitzgerald and Yucas (2005) \cite{Fitzgerald 2005} discovered a
relationship between the factorization of cyclotomic polynomials and that of Dickson polynomials of the first and second kind. This provides us with an alternative method to factor a Dickson polynomial when we know the factorization of the corresponding cyclotomic polynomial. 
However, the problem of the explicit factorization of cyclotomic polynomials over finite fields still remains open.
 
We now give a brief survey of some of the past accomplishments regarding the
factorization of cyclotomic polynomials over finite fields; these are especially
related to our quest to factor $\Phi_{2^nr}.$  
The factorization of $\Phi_{2^n}$ over $\F$ when $q \equiv 1 \pmod{4}$ can be
found for example in \cite{Lidl} and is stated here in Theorem \ref{2^n and q =
1 mod 4}; the more difficult case when $q \equiv 3 \pmod{4}$ was achieved
in 1996 by Meyn \cite{Meyn}. More recently, Fitzgerald and Yucas (2007) \cite{Fitzgerald
2007} gave the factorization of $\Phi_{2^nr}$ over $\F$ for the special cases
where $r$ is an odd prime and $q \equiv \pm 1 \pmod{r}$ is odd. As a
result, the factorizations over $\F$ of $\Phi_{2^n3},$ and the Dickson polynomials of the
first and second kind $D_{2^n3},\ E_{2^n3-1},$ respectively, are thus
obtained. In 2011, L. Wang and Q. Wang
\cite{Prof} went a step further and gave the factorization of $\Phi_{2^n5}$ over $\F.$

In this paper we obtain the complete factorization
of $\Phi_{2^nr}$ over $\F$ for arbitrary $r \geq 3$ odd and $q$ odd such that
$\gcd(q,r) = 1.$ Thus, we generalize the results in \cite{Fitzgerald 2007} and \cite{Prof}. 
We make the assumption that the explicit factorization
of $\Phi_r$ is given to us as a known. When $q = p$ and $r$ is an odd prime
(distinct from $p$) one may use for instance the results due to Stein (2001) \cite{Stein} to
compute the factors of $\Phi_r$ efficiently. We achieve our result by applying the theory of composed products as well as by using, 
and refining in some cases, some of
the techniques and results in \cite{Prof} now generalized for
arbitrary odd number $r > 1.$ In particular, we refine the following result of
theirs. Let $v_2(k)$ denote the highest power of $2$ dividing $k.$

\begin{thm}[{\bf Theorem 2.2, \cite{Prof}}]\label{L}
Let $q = p^s$ be a power of an odd prime $p,$ let $r \geq 3$ be any odd number
such that $\gcd(r,q) = 1,$ and let $L:= L_{\phi(r)} =
v_2\left(q^{\phi(r)} - 1\right)$ be the highest power of $2$ dividing
$q^{\phi(r)} - 1.$ For any $n \geq L$ and any irreducible factor $f$ of
$\Phi_{2^nr}$ over $\F,$ $f(x^{2^{n-L}})$ is also irreducible over $\F.$
Moreover, all irreducible factors of $\Phi_{2^nr}$ are obtained in this way.
\end{thm}

This result implies that if the factorization of $\Phi_{2^Lr}$ is known, then
for $n > L$ we can obtain the factorization of $\Phi_{2^nr}$ by
simply applying the substitution $x \rightarrow x^{2^{n-L}}$ to each
of the irreducible factors of $\Phi_{2^Lr}.$ Thus it only remains to factor
$\Phi_{2^nr}$ when $1 \leq n \leq L.$ We improve the result stated above
by giving a smaller bound $K = v_2(q^{d_r} - 1) \leq L,$ when $d_r = \ord_r(q)$ is
even or $q \equiv 1 \pmod{4};$ here $K$ has the same properties as $L$ just
described, i.e. if the factorization of $\Phi_{2^Kr}$ is known, then for
$n > K$ we obtain the factorization of $\Phi_{2^nr}$ by applying the
substitution $x \rightarrow x^{2^{n-K}}$ to each of the irreducible factors of
$\Phi_{2^Kr}.$ In the case $d_r$ is odd and $q \equiv 3 \pmod{4},$ we show
that the corresponding bound is $v_2(q+1) < L.$ Consequently, it only remains to factor $\Phi_{2^nr}$ when 
$1 \leq n \leq K $ (or $v_2(q+1)$) $\leq L.$ Moreover, we show that
$K$ and $v_2(q+1)$ are the smallest such bounds can be in these cases.

In order to obtain the irreducible factors when $1 \leq n \leq L,$ the 
authors of \cite{Prof} employed the properties $\Phi_{2r}(x) =
\Phi_{r}(-x),$ and $\Phi_{2^{n}r}(x) = \Phi_{2^{n-1}r}(x^2),\ n > 1,$ of
cyclotomic polynomials, together with an iteration of $L$ steps that consists of the
following strategy: 
\\
\\
\tab 1. Obtain the factorization for $n = 0,1.$\\ 
\tab 2. For $1< n \leq L$ and each
irreducible factor $h_{n-1}(x)$ of $\Phi_{2^{n-1}r}(x),$ factor $h_{n-1}(x^2)$
into irreducibles; \tab  these are all the irreducible factors of
$\Phi_{2^{n}r}(x).$\\ \tab \tab If $n = L,$ stop.
\\

First, note that since $q
> 1$ is odd, we may write $q = 2^Am \pm 1,$ for some $A \geq 2,$ and some $m$
odd. Some of our improvements to the above are as follows: In the case $n \leq
A$ or $d_r = \ord_r(q)$ odd, we give the explicit factorization of
$\Phi_{2^nr}$ without the need of any iterations. On the other hand, in the case $d_r$ is even and $n > A,$ we use a
similar strategy to step 2, where we replace $L$ by $K.$ We
show that in the case $d_r$ even it is enough to iterate for at most $v_2(d_r) <
L$ steps starting at $n = A.$ This is quite significant as $L = A + v_2(\phi(r)),$
and so if $A$ is large, say when $q = 2^A - 1$ is a large Mersenne prime, then $L$ is large. 
However, as discussed, we only need to iterate for at most $v_2(d_r)$ steps
which is relatively much smaller. We remark that, similarly as done in
\cite{Prof}, whenever $d_r$ is even or $q \equiv 1 \pmod{4}$ the factorization
of $\Phi_{2^nr}$ can also be formulated in terms of a system of non-linear recurrence relations for $n \leq
K.$ For small finite fields and small $d_r,$
this can be computed fairly fast.

As the reader can infer from the
previous discussion on the properties of the bounds $K$ and $v_2(q+1)$, the
irreducible factors of these cyclotomic polynomials $\Phi_{2^nr}$ are sparse polynomials with a relatively small fixed amount 
of non-zero coefficients and a relatively much higher (as high as
needed) degree. For applications
of sparse polynomials in LRS, efficient implementation of LFSR, and in finite
field arithmetic, see for instance \cite{Berlekamp}, \cite{Golomb}, and \cite{Blake}. Moreover, as another consequence to our factorization, we obtain infinite families of irreducible polynomials. 

We show in Section 3.1 that cyclotomic polynomials are composed
multiplications of other cyclotomic polynomials of lower order. In particular,
$\Phi_{2^nr} = \Phi_{2^n} \odot \Phi_r.$ As a result, we now have at our
disposal additional tools such as the results due to Brawley and Carlitz
(1987) \cite{Brawley and Carlitz} which we quote in Section 2.1; these are
instrumental to our results. We remark that none of the previous authors listed above in our
survey considered this insight. Let $n = p_1^{e_1}p_2^{e_2} \cdots p_s^{e_s}$ be the factorization of $n \in \mathbb{N}$ into powers of distinct primes $p_i,\ 1\leq i \leq s.$ In the case that the orders of $q$ modulo all these prime powers $p_i^{e_i}$ are pairwise
coprime, in Theorem \ref{cyclotomics are composed} we show how to obtain the
factorization of $\Phi_{n}$ from the factorizations of each $\Phi_{p_i^{e_i}}.$
In Theorem \ref{cyclo and minimal} we demonstrate
how to obtain the factorization of $\Phi_{mn}$ from the factorization of
$\Phi_n$ when $q$ is a primitive root modulo $m$ and $\gcd(m,n) =
\gcd(\phi(m),\ord_n(q)) = 1.$

Note that if $S = \{s_k\},\ T = \{t_k\},$ are homogeneous LRS's with
characteristic polynomials $\Phi_{2^n},\ \Phi_r,$ respectively, then the
characteristic polynomial of $ST = \{s_k t_k\}$ is $\Phi_{2^nr} = \Phi_{2^n}
\odot \Phi_r$ by our previous discussion on composed products. We obtain
that for $n$ strictly greater than the corresponding bound $K$ or $v_2(q+1),$ the linear
complexity of such $ST$ is of the form $2^{z(n)}d_r$ where $z(n) = n-K$ or $z(n)
= n - v_2(q+1) + 1,$ respectively. Thus by letting $n \rightarrow \infty,$ the
LRS $ST$ will have a linear complexity approaching infinity. As previously
discussed, this is very desirable in stream cipher theory.

The rest of this paper goes as follows. In Section 2.1 we discuss a few more
properties of composed products and show that some cases of the
Kyuregyan-Kyureghyan's construction are composed products.
In Section 2.2 we give some results regarding the constructions of irreducible polynomials; for this we made use of a theorem on the irreducibility of composed products, due to Brawley and Carlitz (1987). 
We consider Theorem \ref{thm 3} our main result in this section. As a corollary, this generalizes a result due to
Varshamov (1984). As another consequence to Theorem \ref{thm
3}, in Theorem \ref{cyclo and minimal} we show how to obtain the factorization of $\Phi_{mn}$
from the factorization of $\Phi_n$ when $q$ is a primitive root modulo $m$ and $\gcd(m,n) = \gcd(\phi(m),\ord_n(q)) = 1.$
In Section 3.1 we give a number of
results which we later use in order to obtain the factorization of
$\Phi_{2^nr}.$ In Sections 3.2 and 3.3 we give the factorization of
$\Phi_{2^nr}$ over $\F$ when $q \equiv 1 \pmod{4}$ and $q \equiv
3 \pmod{4},$ respectively. Finally in Appendix A we give a table of examples for
Theorem \ref{thm 3} and two tables of examples in Appendix B testing the
recurrence relations in Theorems \ref{2^nr and q = 1 mod 4} and \ref{2^nr and q = 3 mod 4}.

\section{Irreducible Composed Products and Cyclotomic Polynomials}\label{section:irredCon}

In this section we apply Theorem \ref{thm 1}, due to Brawley and Carlitz
\cite{Brawley and Carlitz}, in the construction of new classes of irreducible
polynomials of higher degrees from irreducible polynomials of lower degrees. We devote most of our attention to polynomials of the form $f \odot \Phi_n.$ We consider Theorem \ref{thm 3} our main result in this section.
As a corollary, this generalizes a result due to Varshamov (1984)
\cite{Varshamov}. As another consequence to Theorem \ref{thm 3} we show in
Theorem \ref{cyclo and minimal} how to obtain the factorization of $\Phi_{mn}$ from the
factorization of $\Phi_n$ when $q$ is a primitive root modulo $m$ and $\gcd(m,n) = \gcd(\phi(m),\ord_n(q)) = 1.$
First, in Section 2.1 we give a number of known results in the theory of
composed products which are instrumental.

\subsection{Composed Products}
We need the following known results regarding composed products. 

\begin{prop}[{\bf \cite{Brawley et al}}]\label{comput}
Let $f,\ g \in \F[x].$ Then 
$$\left(f\odot g\right)(x) =
\prod_\alpha \alpha^{n}g\left(\alpha^{-1}x\right)$$
and
$$\left(f\oplus g\right)(x) = \prod_\alpha g\left(x - \alpha\right)$$
where the products $\prod_\alpha$ run over all the roots $\alpha$ of $f.$
\end{prop}

\begin{proof}
\begin{eqnarray*} 
\left(f\odot g\right)(x) &=&
\prod_\alpha \prod_\beta \left(x - \alpha \beta\right) =
\prod_\alpha \prod_\beta \alpha\left(\alpha^{-1}x -
\beta\right)
 =
\prod_\alpha \alpha^{n} g\left(\alpha^{-1}x\right).\\\\
\left(f \oplus g\right)(x) &=& \prod_\alpha \prod_\beta \left(x -
\left(\alpha + \beta\right)\right) 
= \prod_\alpha \prod_\beta \left(\left(x - \alpha \right) -
\beta\right)
 = \prod_\alpha g\left(x - \alpha\right).\qedhere
\end{eqnarray*}
\end{proof}

\begin{prop}[{\bf \cite{Brawley and Carlitz}}]\label{product of composed} 
Let $f_i,\ 1\leq i \leq s,\ g_j,\ 1 \leq j \leq t,$ be polynomials over
$\F.$ Then 
$$\left(\prod_i f_i \odot \prod_j g_j\right) = \prod_i \prod_j
\left(f_i \odot g_j\right).$$
\end{prop}

As we remarked earlier, $(G, \diamond)$ is an abelian group 
when $\diamond$ is the composed multiplication $\odot$, composed sum $\oplus$,
or circle product $\otimes.$ Theorem~\ref{thm 1x} therefore deduces the
following consequence.

\begin{thm}[\cite{Brawley and Carlitz}]\label{thm 1}
Let $f,\ g \in \F[x]$ of degree $m,\ n,$ respectively. Then $f \odot g$,
$f \oplus g$, $f \otimes g$ are irreducible over $\F$ if and only if $f,\ g$ are
irreducible over $\F$ and $\gcd(m,n) = 1.$
\end{thm}

Now we show that some cases of the construction in Theorem~\ref{kk thm} are in
fact composed products and therefore consequences of Theorem~\ref{thm 1}. 

\begin{prop}
Let $\gcd(k, d) =1$, and  $f$ be an irreducible polynomial of degree $k$ over $\F$. Further let $\alpha \neq 0$ and $\beta$ be elements of $\mathbb{F}_{q^d}$. Set $g(x) := f(\alpha x + \beta)$ and let 
\[
F = \prod_{u=0}^{d-1} g^{(u)}
\]
be a polynomial over $\F$ of degree $n = dk$. Then

(i) if $\alpha \in \F$ and $\F(\beta) = \mathbb{F}_{q^d}$, then $F$ is  a
composed sum of two irreducible polynomials with degrees $k$ and $d$
respectively, hence irreducible.

(ii) if $\beta \in \F$ and $\F(\alpha) = \mathbb{F}_{q^d}$, then $F$ is  a
composed multiplication of two irreducible polynomials with degrees $k$ and $d$
respectively, hence irreducible.

(iii) if $\F(\alpha) = \mathbb{F}_{q^d}$ and $\beta = c \alpha,$ where $c \in
\F,$ then $F$ is the result of a linear substitution operation $x \rightarrow (x
+ c)$ applied to an irreducible composed multiplication, and hence
irreducible.


(iv) if  $\alpha = -\beta  + 1$  and  $\F(\alpha, \beta) = \mathbb{F}_{q^d},$ then
$F$ is the circle product of two irreducible polynomials with degrees $k$ and
$s$ respectively, where $s\mid d,$ hence irreducible.

(v) if  $\alpha = \beta  + 1$ and  $\F(\alpha, \beta) = \mathbb{F}_{q^d},$ then
$F$ is the composed product of two irreducible polynomials with degrees $k$ and
$s$ respectively, where $s\mid d,$ hence irreducible.

\end{prop} 

\begin{proof}
(i) Because $\alpha \in \F$, we write $\bar{f}(x) = f(\alpha x)$. So $\bar{f}
(x)$ is also an irreducible polynomial of degree $k$ over $\F$. Therefore, by
Proposition~\ref{comput},
\[
F(x) = \prod_{u=0}^{d-1} f^{(u)}(\alpha x + \beta) = \prod_{u=0}^{d-1} \bar{f}^{(u)}( x + \alpha^{-1} \beta) 
\]
is the composed sum of $\bar{f}$ and the minimal polynomial of $\alpha^{-1} \beta$ (an irreducible polynomial of degree $d$).

(ii) In this case, let $\bar{f}(x) = f(x + \beta)$. So $\bar{f} (x)$ is also
an irreducible polynomial of degree $k$ over $\F$.
\[
F(x) = \prod_{u=0}^{d-1} f^{(u)}(\alpha x + \beta) = \prod_{u=0}^{d-1}
\bar{f}^{(u)}( \alpha x).
\]
Hence all the roots of $F$ are the product of roots of $\bar{f}$ and roots of
the minimal polynomial of $\alpha^{-1};$ moreover, both are irreducible
polynomials over $\F.$ Therefore $F$ is the irreducible composed multiplication
of $\bar{f}$ and the minimal polynomial of $\alpha^{-1}$ (both have coprime
degrees).

(iii) Note that $\prod_{u=0}^{d-1} \alpha^{-kq^u}f\left(\alpha^{q^u} x\right)$
is an irreducible composed multiplication over $\F.$ Thus, since
$\prod_{u=0}^{d-1}\alpha^{-kq^u} \in \F^*,$ it must be that 
$$
H(x) = \prod_{u=0}^{d-1} f\left(\alpha^{q^u} x\right) = \prod_{u=0}^{d-1}
f^{(u)}(\alpha x) 
$$
is irreducible as well over $\F.$ But then 
$$
H(x + c) = \prod_{u=0}^{d-1} f^{(u)}(\alpha(x + c)) = \prod_{u=0}^{d-1}
f^{(u)}(\alpha x + \beta) = F(x)
$$
is irreducible over $\F.$

(iv) Let $h$ be the minimal polynomial of $-\alpha^{-1} + 1$. Because $\F(\alpha, \beta) = \mathbb{F}_{q^d}$, there are $s \mid d$ distinct conjugates of $- \alpha^{-1} + 1$ and thus the degree of $h$ is $s$.  We denote an
arbitrary root of $f$ and $h$ by $\alpha_f$ and $\alpha_h$ respectively. Then an arbitrary root of $ F(x) =  \prod_{u=0}^{d-1} f^{(u)}(\alpha x + \beta) 
$
can be written as 
\[
\alpha^{-1} (\alpha_f - \beta) = \alpha^{-1} (\alpha_f + \alpha - 1) =  \alpha^{-1} \alpha_f  + 1  -  \alpha^{-1} = (1 - \alpha_h) \alpha_f + \alpha_h = \alpha_f + \alpha_h - \alpha_f \alpha_h.
\]
Because $h$ has degree $s \mid d$ as a consequence of $\F(\alpha, \beta) =
\mathbb{F}_{q^d},$ the polynomial $F$ is the composed product of two irreducible
polynomials of coprime degrees, and hence irreducible.

(v) Here we define the composed product $\diamond$ for $G = \Gamma_{q} - \{
-1\}$ by $a \diamond b = a + b + ab,$ which forms an abelian group similar to
the group corresponding to the circle product. Similarly, let $h$ be the minimal
polynomial of $\alpha^{-1} - 1$ and denote an arbitrary root of $f$ and $h$ by $\alpha_f$ and $\alpha_h$ respectively. Then an arbitrary root of $ F(x) =  \prod_{u=0}^{d-1} f^{(u)}(\alpha x + \beta) $ can be written as 
\[
\alpha^{-1} (\alpha_f - \beta) = \alpha^{-1} (\alpha_f - \alpha +1) =  \alpha^{-1} \alpha_f - 1 + \alpha^{-1} = (1+\alpha_h) \alpha_f + \alpha_h = \alpha_f + \alpha_h + \alpha_f \alpha_h.
\]
Because $h$ has degree $s \mid d$ as a consequence of $\F(\alpha, \beta) =
\mathbb{F}_{q^d},$ the polynomial $F$ is the composed product of two irreducible
polynomials of coprime degrees, and hence irreducible.
\end{proof}

\subsection{Irreducible Constructions}

In this subsection we use the composed multiplication to construct some new
classes of irreducible polynomials.

\begin{lem}\label{lem 1}
Let $f$ be an irreducible polynomial over $\F$ of degree $n$
belonging to order $t$, and let $r$ be a positive integer. Then $f(x)\mid  
f(x^r)$ implies $r \equiv q^i \pmod{t}$ for some $i \in [0,n-1]$. Furthermore, let
$\alpha$ be a root of $f$ and assume $r \equiv q^i \pmod{t}$ as above. Then the sets 
\[
R = \{\alpha^{r^kq^u};\ 0\leq u \leq n-1\},\ F = \{\alpha^{q^u};\ 0\leq u \leq
n-1\}\] 
are equal for any $k\geq 0$.
 \end{lem}
 
 \begin{proof}
 Recall that the roots of $f$ are $\alpha^{q^u}$, $0\leq u\leq n-1$,
 and $q^n \equiv 1 \pmod{t}$ because $t\mid   q^n-1$. Moreover, note that $q^n
 \equiv 1 \pmod{t}$ implies that for any $m \geq 0$ there exists an $s \in [0,n-1]$ such
 that $q^m \equiv q^s \pmod{t}$. We have: $f(x)\mid  f(x^r)$ implies
 $f(\alpha^{q^ur}) = 0$ for all $u \in [0,n-1]$ giving $\alpha^{q^ur} =
 \alpha^{q^j}$, some $j \in [0,n-1];$ hence $q^ur \equiv q^j \pmod{t}$
 and so $r \equiv
 q^{n+j-u} \equiv q^i \pmod{t},$ some $i \in [0,n-1].$
 
 Next, assume $r \equiv q^i \pmod{t}$ for some $i\in [0,n-1]$. We show that $R =
 F$. Clearly, $r^kq^u \equiv q^{ik+u} \equiv q^j \pmod{t}$ for some $j\in
 [0,n-1].$ Thus, $\alpha^{r^kq^u} = \alpha^{q^j} \in F;$ hence $R \subseteq F.$ 
 Now let $\alpha^{q^u} \in F$. Note that $r\equiv q^i\pmod{t}$ implies $r^k
 \equiv q^l \pmod{t}$ for some $l\in [0,n-1].$ If $u \geq l$, then $r^kq^{u-l}
 \equiv q^u \pmod{t}$, so $\alpha^{q^u} = \alpha^{r^kq^{u-l}} \in R$. If $u < l$,
 write $r^k \equiv q^{u+s}\pmod{t},$ where $0 < s = l - u \leq n-1$. Then
 $r^kq^{n-s} \equiv q^{u+s+(n-s)} \equiv q^u\pmod{t},$ and hence $\alpha^{q^u} =
 \alpha^{r^kq^{n-s}} \in R$. Therefore $R = F.$
 \end{proof}

 \begin{lem}[\bf{Exercise 10.12, \cite{Wan}}]\label{lem 2} Let $r$ be an odd
 prime number and $q$ a prime power. Suppose that $q$ is a primitive root modulo
 $r$ and $r^2\nmid \left(q^{r-1}-1\right)$. Then the polynomial $$\Phi_r(x^{r^k}) = x^{(r-1)r^k} + x^{(r-2)r^k} + \dots + x^{r^k} + 1$$
 is irreducible over $\F$ for each $k \geq 0$.
 \end{lem}
 
 \begin{proof}
 First, recall that the hypotheses imply that $q$ is a primitive root modulo
 $r^k$, $k\geq 1$. Then $\Phi_{r^{k+1}},\ k\geq 0$, is irreducible over $\F$.
 Thus, if we show $\Phi_{r^{k+1}}(x) = \Phi_r(x^{r^k}),$ the result is achieved.
 Indeed,
 \[\Phi_{r^{k+1}}(x) =
 \prod_{d\mid  r^{k+1}}\left(x^{r^{k+1}/d}-1\right)^{\mu(d)} =
 \frac{x^{r^{k+1}}-1}{x^{r^k}-1} = \Phi_r\left(x^{r^k}\right). \qedhere\]
 \end{proof}
 
 The following result is the construction of a new infinite family of
 irreducible polynomials over $\F.$
 
 \begin{thm}\label{thm 2}
 Let $r$ be a prime number and let $f$ be an irreducible
 polynomial over $\F$ of degree $n$ such that
 
 \begin{tabular}{ c l}
 
 (i) & $f(x)\mid  f(x^r)$\\ 
 (ii) & $q$ is a primitive root modulo $r$\\ 
(iii) & $\gcd(n,r-1) = 1$.\\
 \end{tabular}
\\\\
We have:\\
\tab (a) The polynomial $F(x) = f(x^r)\left(f(x)\right)^{-1} = \left(f\odot
\Phi_r\right)(x)$ is irreducible over $\F$ of degree $n(r-1)$.\\
\tab (b) If $r$ is an odd prime such that $r^2\nmid
\left(q^{r-1}-1\right)$ and $\gcd(n,r(r-1)) = 1,$ then $F\left(x^{r^k}\right)
= \left(f\odot \Phi_{r^{k+1}}\right)(x),\ k \geq 0,$ is an irreducible
polynomial over $\F$ of degree $nr^k(r-1).$
\end{thm}

 \begin{proof}
 (a) Condition (i) and Lemma \ref{lem 1} imply that
 \[
 f(x) = \prod_{u=0}^{n-1}\left(x-\alpha^{q^u}\right) =
 \prod_{u=0}^{n-1}\left(x-\alpha^{rq^u}\right).
 \]
 As a result,
 \[
 F(x) = \frac{f(x^r)}{f(x)} =
 \prod_{u=0}^{n-1}\left(\frac{x^r-\alpha^{rq^u}}{x-\alpha^{q^u}}\right).
 \]
 Note that
 \[
 \frac{x^r-\alpha^{rq^u}}{x-\alpha^{q^u}} = x^{r-1} +\alpha^{q^u} x^{r-2} +
 \dots +\alpha^{(r-1)q^u} = \alpha^{(r-1)q^u}\Phi_r\left(\alpha^{-q^u}x\right).
 \]
 Condition (ii) implies that $\Phi_r$ is irreducible over $\F$ of
 degree $r-1$ which is coprime to $n$ by condition (iii). It only remains to
 observe that $$F(x) = \prod_{u=0}^{n-1}\alpha^{(r-1)q^u}\Phi_r\left(\alpha^{-q^u}x\right)
 = \left(f\odot \Phi_r\right)(x)$$
 by Proposition \ref{comput}. Now Theorem \ref{thm 1} completes the proof of
 (a).\\
 
 We now prove (b): Lemma \ref{lem 2} gives $\Phi_r\left(x^{r^k}\right) =
 \Phi_{r^{k+1}}(x)$ is irreducible over $\F$ of degree $r^k(r-1)$ which
 is coprime to $n$ by assumption. By condition (i), Lemma \ref{lem 1}, and
 Proposition \ref{comput}, we obtain
 \begin{eqnarray*} 
 F\left(x^{r^k}\right)
 &=& \prod_{u=0}^{n-1}\alpha^{(r-1)q^u}\Phi_r\left(\alpha^{-q^u}x^{r^k}\right)
 =
 \prod_{u=0}^{n-1}\alpha^{r^k(r-1)q^u}\Phi_r\left(\alpha^{-r^kq^u}x^{r^k}\right)\\
 &=&
 \prod_{u=0}^{n-1}\alpha^{r^k(r-1)q^u}\Phi_{r^{k+1}}\left(\alpha^{-q^u}x\right)
 = \left(f\odot \Phi_{r^{k+1}}\right)(x).
  \end{eqnarray*}
 Noting that
 $f\odot \Phi_{r^{k+1}}$ is irreducible over $\F$ of degree $nr^k(r-1)$ by
 Theorem \ref{thm 1}, we thus obtain the result.
 \end{proof}
 
 \begin{eg} We give an example where conditions (i), (ii), (iii) are satisfied. 
 As shown in Lemma
 \ref{lem 1}, if $f(x)\mid  f\left(x^r\right)$, then $r \equiv q^i\pmod{t}$ for some $i \in [0,n-1],$ where $t$ is the order of $f.$ Moreover, we need
 $\ord_t(q) = n$ (see Lemma \ref{lem 3}), $\ord_r(q) = \phi(r)$ and
 $\gcd(n,\phi(r)) = 1.$ The reader can verify that when $\left(q,n,t,r,f(x)\right) = \left(2,3,7,11,x^3+x^2+1\right)$
 all the conditions are met. Furthermore, $11^2\nmid\left(2^{10}-1\right)$ and
 $\gcd(3,11 \cdot 10) = 1,$ so part (b) also holds in this case.
 \end{eg}
 
 We generalize the last result further in the following theorem. This also
 generalizes a result due to Varshamov (1984) which we state in Corollary
 \ref{varshamov}. We need the following well known fact.
 
 \begin{lem}[{\bf Theorem 3.5, \cite{Lidl}}]\label{lem 3}
 Let $f$ be an irreducible polynomial over $\F$ of degree $n$ belonging to
 order $t$. Then the multiplicative order of $q$ modulo $t$ is $n$.
 \end{lem}
 
 \begin{thm}\label{thm 3} Let $m \in \mathbb{N}$ and assume that $q$ is a
 primitive root modulo $m$. Let $f$ be an irreducible polynomial over $\F$ of degree $n$ such that
$\gcd(n,\phi(m)) = 1$ with $f$ belonging to order $t$. If $m$ and $t$ are
even, further assume that $n$ is the multiplicative order of $q$ modulo $t/2.$
For each positive divisor $d$ of $m$ define the polynomials $R_d$, $\Psi_d $ over $\F$ as follows: Set $x^d \equiv R_d(x)\pmod {f(x)}$, 
and $\Psi_{d}(x) = \sum_{i=0}^{n}\Psi_{d,i}x^i$, where $\Psi_{d}$ is the
non-zero polynomial of minimal degree satisfying the congruence
\[ \sum_{i=0}^{n}\Psi_{d,i}\left(R_d(x)\right)^i\equiv 0 \pmod {f(x)}.\]
Then the polynomials $\Psi_d,\ d \mid   m,$ are irreducible over $\F$ of degree
$n$. Furthermore,
$$F_m(x) = \prod_{d \mid   m}\Psi_d\left(x^d\right)^{\mu(m/d)} =
\left(f\odot \Phi_m\right)(x)$$ is an irreducible polynomial over $\F$ of
degree $n\phi(m)$ belonging to order $\lcm(t,m)$.
 \end{thm}

\begin{proof}
We first prove that for each positive divisor $d$ of $m,\ \Psi_d$ is
irreducible over $\F$ of degree $n$. Now, let $\alpha \in \Fn$ be a root of
$f.$ Then the congruence relations
$\sum_{i=0}^{n}\Psi_{d,i}\left(R_d(x)\right)^i \equiv 0 \pmod {f(x)}$ and $x^d
\equiv R_d(x) \pmod {f(x)}$ imply that $R_d(\alpha) = \alpha^d$ is a root of $\Psi_d.$
Thus, by the assumption of the minimality of the degree of $\Psi_d$ we deduce
that $\Psi_d$ is the minimal polynomial of $\alpha^d$ over $\F.$ As a result, $\Psi_d$ is irreducible over $\F.$

We now prove $\deg \left(\Psi_d\right) = n.$ Suppose $\deg \left(\Psi_d\right) = s_d
\leq n.$ Note that $\ord\left(\Psi_d\right) = \ord\left(\alpha^d\right) =
t/\gcd(d,t).$ Then by Lemma \ref{lem 3} we have $\ord_t(q) = n,$ and
$\ord_{t/\gcd(d,t)}(q) = s_d.$ Since $q$ is a primitive root modulo $m,$ then $m$ must be either $1,\ 2,\ 4,\
r^k,$ or $2r^k$ for some odd prime $r$ and some $k\geq 1.$ We show that in all
these cases $s_d = n$ for each $1\leq d \mid   m.$ Observe that $\Psi_1$ is the minimal polynomial of $\alpha$ which is $f$; hence $\Psi_1 = f$ and $s_1 = n.$ 
Suppose $d = 2\mid   m.$ If $\gcd(d,t) = 1,$ then $s_2 = \ord_{t/\gcd(2,t)}(q) =
\ord_{t}(q) = n.$ Otherwise $t$ is even and so 
$s_2 = \ord_{t/\gcd(2,t)}(q) = \ord_{t/2}(q) = n$ by the hypothesis for $m$
even. Note that whenever $m > 2$ we can't have $\gcd(m,t) = m$ otherwise $q^n
\equiv 1 \pmod{t}$ gives $q^n \equiv 1 \pmod{m}$ implying $\phi(m)\mid  n$
contrary to $\gcd(n,\phi(m)) = 1$ and $\phi(m) > 1.$ Thus whenever $m = 4$ we must have
either $\gcd(m,t) = 1$ or $\gcd(m,t) = 2.$ In both cases we obtain 
$s_4 = \ord_{t/\gcd(4,t)}(q) = \ord_{t}(q) = n$ or $s_4 = \ord_{t/2}(q) = n$
also by the hypothesis for $m$ even. Consider the cases $m = r^k,\ 2r^k,$ for some odd prime $r,$ some $k\geq 1.$ Let $d = r^j\mid  m,\ 1\leq j \leq k.$ Either $r\mid  \gcd \left(r^j,t\right)$ or $\gcd
\left(r^j,t\right) = 1.$ Suppose $r\mid  \gcd \left(r^j,t\right).$ In particular,
$r\mid  t.$ Note that $\phi(m) > 1$ is even and so the assumption $\gcd(n,\phi(m)) =
1$ implies $n$ is odd. Moreover, because $q$ is a primitive root modulo $m = r^k$ or $2r^k,$ then $q$
is a primitive root modulo $r.$ Now, $q^n \equiv 1
\pmod{t}$ gives $q^n \equiv 1 \pmod{r}$ implying $\phi(r) = r-1\mid  n.$ But $n$ is odd
and $r-1$ is even because $r$ is odd. Thus we have reached a contradiction and
so we must have $\gcd\left(r^j,t\right) = 1$. As a result we obtain $s_{r^j} =
\ord_{t/\gcd\left(r^j,t\right)}(q) = \ord_{t}(q) = n.$ At this point we have
accounted for all possible positive divisors $d$ of $m$ and we thus conclude
$s_d = n$ for each $1\leq d \mid   m;$ therefore
$$\Psi_d(x) = \prod_{u=0}^{n-1}\left(x - \alpha^{dq^u}\right).$$

Now, we know that $\Phi_m$ is irreducible over $\F$ since $q$ is a primitive
root modulo $m.$ Moreover, $\deg \left(\Phi_m\right) = \phi(m)$ is coprime to
$n$ by assumption. Thus, by Theorem \ref{thm 1}, $f\odot
\Phi_m$ is irreducible over $\F$ of degree $n\phi(m).$ Furthermore, because the roots
$\{\xi_m\}$ of $\Phi_m$ are the primitive $m$-th roots of unity, i.e., $m$ is
the least positive integer $l$ such that $\xi_m^l = 1,$ then
$\ord\left(\xi_m\right) = m.$ Hence, $\ord\left(f\odot
\Phi_m\right) = \ord\left(\alpha \xi_m\right) = \lcm(t,m).$ In
conclusion, if we show $F_m = f\odot \Phi_m,$ the proof will
be complete. First, recall
$$x^m-1 = \prod_{k=0}^{m-1}\left(x - \xi_m^k\right) =
\prod_{d \mid   m}\Phi_d(x) = \prod_{d \mid   m}\prod_{\substack{k=0\\
\gcd(k,d)=1}}^{d-1}\left(x-\xi_d^k\right).$$ 
We have
\begin{eqnarray*}
\Psi_m\left(x^m\right) &=& \prod_{u=0}^{n-1}\left(x^m-\alpha^{mq^u}\right) =
\prod_{u=0}^{n-1}\prod_{k=0}^{m-1}\left(x-\alpha^{q^u}\xi_m^k\right) 
= \left(f\odot \left(x^m-1\right) \right)(x)
=
\left(f\odot \prod_{d \mid   m} \Phi_d\right)(x)\\
 &=&
\prod_{u=0}^{n-1}\prod_{d \mid   m}\prod_{\substack{k=0\\
\gcd(k,d)=1}}^{d-1}\left(x-\alpha^{q^u}\xi_d^k\right)
 =
\prod_{d \mid   m}\prod_{u=0}^{n-1}\prod_{\substack{k=0\\
\gcd(k,d)=1}}^{d-1}\left(x-\alpha^{q^u}\xi_d^k\right)\\ 
&=& 
\prod_{d \mid   m}\left(f\odot \Phi_d\right)(x).
\end{eqnarray*}
By applying the Mobius Inversion Formula now we obtain the desired result.
\end{proof}

\begin{rem}
Whenever the hypotheses in Theorem \ref{thm 3} are true, the proof shows, in
particular, that the characteristic polynomial of each $\alpha^d,\ 1\leq d \mid   m,$ is its
minimal polynomial, and thus it is irreducible. Note that the condition ``If $m$
and $t$ are even, further assume that $n$ is the multiplicative order of $q$
modulo $t/2$'' is necessary to ensure that for any even positive divisor $d$ of $m,$
the characteristic polynomial of $\alpha^d$ is irreducible; this is true in
most cases here. However, the reader can observe from the proof that if we
define $\Psi_d$ as the characteristic polynomial of $\alpha^d$ instead, $F_m$
will still be irreducible.
\end{rem}

\begin{rem}
Note that since $m$ is either of $1,\ 2,\ 4,\ r^k,\ 2r^k,$ and $\mu(c) = 0$
whenever there exists some prime $p$ such that $p^2\mid  c,$ then any $F_m$ must be a
product and division of at most four minimal polynomials
$\Psi_d$ evaluated at $x^d.$ Since one of these must be the
given $\Psi_1 = f,$ we only need to compute at most three minimal (or characteristic, see above) polynomials. Thus, this may provide an alternative more
efficient way to compute $f \odot \Phi_m$ versus other
known general methods for computing composed products. See \cite{Brawley et
al} for known methods of computing composed products
efficiently. We further remark that our formula $F_m = f \odot \Phi_m$ holds even if $\gcd(n,\phi(m)) \neq 1,$ although $F_m$ is not irreducible in this case.
\end{rem}

\begin{rem}\label{rem 3} 
Theorem \ref{thm 2} (a) is a corollary of Theorem \ref{thm 3}. Indeed, 
$$F(x) = \frac{f\left(x^r\right)}{f(x)} = \left(f\odot
\Phi_{r}\right)(x) = F_r(x).$$
\end{rem}

Theorem \ref{thm 3} generalizes a result due to Varshamov (1984) which was 
given without a proof. For an independent proof of Corollary \ref{varshamov} we
refer the reader to Theorem 3 in \cite{Kyuregyan}.

\begin{cor}[\bf{Varshamov (1984)}]\label{varshamov}
Let $r$ be an odd prime number which does not divide $q$ and $r-1$ be the order
of $q$ modulo $r.$ Further, let $n \in \mathbb{N}$ such
that $\gcd(n,r-1) = 1,$ and let $f$ be an irreducible polynomial of degree $n$ over $\F$
belonging to order $t.$ Define the polynomials $R$ and $\psi$ over $\F$ as follows: Set $x^r \equiv
R(x)\pmod {f(x)}$ and $\psi(x) = \sum_{u=0}^{n}\psi_ux^u,$ where $\psi$ is the
nonzero polynomial of minimal degree satisfying the congruence 
$$\sum_{u=0}^{n}\psi_u\left(R(x)\right)^u \equiv 0 \pmod {f(x)}.$$
Then the polynomial $\psi$ is an irreducible polynomial of degree $n$ over
$\F$ and
$$F(x) = \left(f(x)\right)^{-1}\psi\left(x^r\right)$$
is an irreducible polynomial of degree $(r-1)n$ over $\F.$ Moreover, $F$
belongs to order $rt.$
\end{cor}

\begin{proof}
In Theorem \ref{thm 3}, let $m = r.$ Then $F_r$ is an irreducible polynomial
over $\F$ of degree $\phi(r)n = (r-1)n$ belonging to order $\lcm(r,t).$ Recall
from the proof of Theorem \ref{thm 3} that if an odd prime $r$ divides $m,$ then
$\gcd(r,t) = 1.$ Thus $F_r$ belongs to order $\lcm(r,t) = rt.$ Let $\alpha$ be a root of $f$. The definition of $\psi$ implies it is 
the minimal polynomial of $\alpha^r$ which is $\Psi_r;$ 
thus $\psi = \Psi_r$ and so $\psi$ is irreducible over $\F$ of degree $n$. 
It only remains to observe 
\[F_r(x) =
\prod_{d\mid  r}\Psi_d\left(x^d\right)^{\mu(r/d)} = \frac{\Psi_r\left(x^r\right)}{\Psi_1(x)} = \frac{\psi\left(x^r\right)}{f(x)} =
F(x).\qedhere \]
\end{proof}

\begin{cor}
Let $r$ be an odd prime and assume $q$ is a primitive root modulo $r$ such that
$r^2\nmid \left(q^{r-1}-1\right).$ Let $f$ be an
irreducible polynomial over $\F$ of degree $n$ such that
$f(x)\mid  f\left(x^r\right)$ and $\gcd\left(n,r(r-1)\right) = 1$. Then for $k \geq
0,$ 
$$F_r\left(x^{r^k}\right) = F_{r^{k+1}}(x)$$ 
is an
irreducible polynomial over $\F$ of degree $nr^k(r-1).$
\end{cor}

\begin{proof}
Let $F(x) = \left(f(x)\right)^{-1}f\left(x^r\right) =
\left(f\odot\Phi_r\right)(x)$ as in Theorem \ref{thm 2}. Then
$F\left(x^{r^k}\right)$ is irreducible over $\F$ of degree $nr^k(r-1)$ by
Theorem \ref{thm 2} (b). It only suffices to note that by Remark \ref{rem 3} and
Theorem \ref{thm 2} (b) we have \[F_r\left(x^{r^k}\right) = F\left(x^{r^k}\right) = \left(f\odot
\Phi_{r^{k+1}}\right)(x) = F_{r^{k+1}}(x). \qedhere \]
\end{proof}

\section{Explicit Factorization of the Cyclotomic Polynomial
$\Phi_{2^nr}$}\label{cyclo fact}

In this section we present new results, Theorems \ref{thm 5},
\ref{2^nr and q = 1 mod 4}, \ref{2^nr and q = 3 mod 4}, of the
explicit factorization of $\Phi_{2^nr}$ over $\F$ where $q$ is odd, $n \in
\mathbb{N},$ and $r \geq 3$ is any odd number such that $\gcd(q,r) = 1$. 
Previously, only $\Phi_{2^n3}$ and $\Phi_{2^n5}$ had been factored
in \cite{Fitzgerald 2007} and \cite{Prof}, respectively. We also show how
to obtain the factorization of $\Phi_n$ in a special case in Theorem
\ref{cyclotomics are composed}, and how to obtain the
factorization of $\Phi_{mn}$ from the given factorization of $\Phi_n$ when $q$
is a primitive root modulo $m$ and $\gcd(m,n) = \gcd(\phi(m),\ord_n(q)) = 1.$

\subsection{Preliminaries}

The following result shows that cyclotomic polynomials are in fact
composed multiplications of other cyclotomic polynomials. Moreover, it shows how we may
obtain the factorization of $\Phi_n$ in a special case.

\begin{thm}\label{cyclotomics are composed}
Let $n = p_1^{e_1}p_2^{e_2}\dots p_s^{e_s}$ be the complete factorization of $n
\in \mathbb{N}.$ Let $\Phi_{p_1^{e_1}} = \prod_i f_{1_i},\
\Phi_{p_2^{e_2}} = \prod_j f_{2_j}, \dots, \ \Phi_{p_s^{e_s}} = \prod_k f_{s_k}$ 
be the corresponding factorizations over $\F.$
Then 
\begin{eqnarray*}
\Phi_n &=& \Phi_{p_1^{e_1}} \odot \Phi_{p_2^{e_2}} \odot
\dots \odot \Phi_{p_s^{e_s}}\\
&=&
\prod_i \prod_j \cdots \prod_k \left(f_{1_i}\odot  f_{2_j}
\odot \cdots \odot f_{s_k}\right).
\end{eqnarray*}
Moreover, if the multiplicative orders of $q$ modulo all these primes powers
$p_i^{e_i}$ are pairwise coprime, then this is the complete factorization of
$\Phi_{n}$ over $\F.$
\end{thm}

\begin{proof}
For brevity's sake, let $F = \Phi_{p_1^{e_1}} \odot \cdots
\odot \Phi_{p_s^{e_s}}.$ By definition,
$$
F(x) = \prod_{\xi_{p_1^{e_1}}} \cdots \prod_{\xi_{p_s^{e_s}}}(x -
\xi_{p_1^{e_1}} \cdots \xi_{p_s^{e_s}})
$$
where the products $\prod_{\xi_{p_i^{e_i}}}$ run over all primitive
$p_i^{e_i}$-th roots of unity $\xi_{p_i^{e_i}}$. Note that each
$\xi_{p_1^{e_1}}\xi_{p_2^{e_2}} \cdots \xi_{p_s^{e_s}}$ is a root of $\Phi_n.$
Indeed, $\ord(\xi_{p_1^{e_1}} \cdots \xi_{p_s^{e_s}}) = p_1^{e_1} \cdots
p_s^{e_s} = n$ as $\ord(\xi_{p_i^{e_i}}) = p_i^{e_i}$ and the $p_i$'s are coprime; thus each $\xi_{p_1^{e_1}} \cdots \xi_{p_s^{e_s}}$ is a primitive
$n$-th root of unity, and hence a root of $\Phi_n.$ Furthermore, both
polynomials are monic and $\deg (F) = \prod_{i=1}^s \phi(p_i^{e_i}) =
\phi(\prod_{i=1}^s p_i^{e_i}) = \phi(n) = \deg \Phi_n.$ Now, recall that all the roots of a cyclotomic polynomial are distinct. If we
show that all roots $\xi_{p_1^{e_1}} \xi_{p_2^{e_2}}\cdots \xi_{p_s^{e_s}}$ of
$F$ are distinct, the desired result $\Phi_n = F$ must then follow. Suppose
$\xi_{p_1^{e_1}}^{i_1} \cdots \xi_{p_s^{e_s}}^{i_s} = \xi_{p_1^{e_1}}^{j_1} \cdots
\xi_{p_s^{e_s}}^{j_s}$ is a root of $F.$ Then $\xi_{p_1^{e_1}}^{i_1-j_1}
\cdots \xi_{p_{s-1}^{e_{s-1}}}^{i_{s-1}-j_{s-1}} = \xi_{p_s^{e_s}}^{j_s-i_s}.$
In particular, $\ord(\xi_{p_1^{e_1}}^{i_1-j_1}
\cdots \xi_{p_{s-1}^{e_{s-1}}}^{i_{s-1}-j_{s-1}})$ = $\ord(\xi_{p_s^{e_s}}^{j_s-i_s}).$ Moreover,
$\ord(\xi_{p_1^{e_1}}^{i_1-j_1}
\cdots \xi_{p_{s-1}^{e_{s-1}}}^{i_{s-1}-j_{s-1}}) \mid p_1^{e_1} \cdots p_{s-1}^{e_{s-1}}$ and $\ord(\xi_{p_s^{e_s}}^{j_s-i_s}) \mid p_s^{e_s}.$ But then, as
$\gcd(p_1^{e_1} \cdots p_{s-1}^{e_{s-1}},\ p_s^{e_s}) = 1,$ we must have
$\xi_{p_s^{e_s}}^{j_s-i_s} = 1.$ Since $p_s^{e_s} > 1$ and $0 < i_s, j_s <
p_s^{e_s},$ necessarily $i_s = j_s.$ Similarly, by induction we can show $i_k = j_k,\ 1 \leq k \leq s.$ Thus,
$\Phi_n = F.$

The second statement of the theorem follows from Proposition \ref{product of
composed}, the associativity of composed multiplications, and Theorem \ref{thm
1} combined with the fact that the degrees of the irreducible factors $f_i$
of $\Phi_{p_i^{e_i}}$ are $\ord_{p_i^{e_i}}(q).$
\end{proof}

\begin{eg}
Let $q = 11,\ n = 595 = 5\cdot 7 \cdot 17.$ As 
$\ord_5(q) = 1,\ \ord_7(q) = 3,\ \ord_{17}(q) = 16$ are pairwise coprime, then
by Theorem \ref{cyclotomics are composed} the complete factorization of
$\Phi_{595}$ over $\mathbb{F}_{11}$ is given by
$$
\Phi_{595} = \prod_i \prod_j \prod_k (f_i \odot g_j \odot h_k)
$$
where the $f_i,\ g_j,\ h_k$ are the irreducible factors of
$\Phi_5,\ \Phi_7,\ \Phi_{17},$ respectively, over $\mathbb{F}_{11}.$
\end{eg}

We have the following corollary to Theorem \ref{cyclotomics are composed}.
\begin{cor}\label{Phi_mn}
Let $m,\ n \in \mathbb{N}$ be coprime. Then
$\Phi_{mn} = \Phi_m\odot \Phi_n$. Further, let $\Phi_{m} = \prod_i f_i, \ \Phi_n = \prod_j g_j$ be the respective
factorizations over $\F$. Then 
$$
\Phi_{mn} = \prod_i \prod_j\left(f_i\odot g_j\right).
$$ 
Moreover, if $\gcd(\ord_m(q),\ord_n(q)) = 1,$ then this is the
complete factorization of $\Phi_{mn}$ over $\F$.
\end{cor}

\begin{proof}
The result is clear if $m = 1$ or $n = 1.$ Assume $m = p_1^{e_1}p_2^{e_2}\cdots
p_k^{e_k},\ n = p_{k+1}^{e_{k+1}} p_{k+2}^{e_{k+2}} \cdots p_s^{e_s}$ are
complete factorizations of $m,\ n$ over $\mathbb{N}.$ Then by Theorem
\ref{cyclotomics are composed} we have
$$
\Phi_{m} = \Phi_{p_1^{e_1}} \odot \cdots \odot \Phi_{p_k^{e_k}}, \tab
\Phi_n = \Phi_{p_{k+1}^{e_{k+1}}} \odot \cdots \odot \Phi_{p_s^{e_s}},
$$
giving 
$$
\Phi_m \odot \Phi_n = (\Phi_{p_1^{e_1}} \odot \cdots \odot
\Phi_{p_k^{e_k}}) \odot (\Phi_{p_{k+1}^{e_{k+1}}} \odot \cdots \odot
\Phi_{p_s^{e_s}}) = \Phi_{p_1^{e_1}} \odot \cdots \odot
\Phi_{p_s^{e_s}} = \Phi_{mn}.$$

The second statement follows immediately from Proposition \ref{product
of composed} and Theorem \ref{thm 1} combined with the fact that the degrees of
the irreducible factors $f_i,\ g_j$ are $\ord_m(q),\ \ord_n(q)$ respectively.
\end{proof}

In particular, whenever $r$ is odd we have $\Phi_{2^nr} = \Phi_{2^n}
\odot \Phi_r.$ Thus whenever the factorizations of $\Phi_m,\ \Phi_n$
are known, and $\gcd(m,n) = \gcd(\ord_m(q),\ord_n(q)) = 1,$ we can obtain all the irreducible factors of $\Phi_{mn}$ by computing each $f_i\odot g_j.$
This is a significant tool in the factorization of polynomials which we will use
frequently in order to obtain some of the following results. 

The following result shows how we may obtain the factorization of
$\Phi_{mn}$ from the factorization of $\Phi_n$ whenever $q$ is a primitive
root modulo $m$ and $\gcd(m,n) = \gcd(\phi(m),\ord_n(q)) = 1.$ Recall that $\Phi_n$ decomposes into $\phi(n)/\ord_n(q)$
irreducible factors over $\F$ of the same degree $\ord_n(q)$ whenever
$\gcd(q,n) = 1.$

\begin{thm}\label{cyclo and minimal}
Let $m,\ n \in \mathbb{N},\ \gcd(m,n) =
\gcd(\phi(m),d_n) = 1,$ where $d_n = \ord_n(q)$. Assume $q$ is a primitive root
modulo $m.$ Let $\Phi_n = \prod_{i=1}^{\phi(n)/d_n} f_i$ be the corresponding
factorization over $\F.$ 
Then the factorization of $\Phi_{mn}$ over $\F$ is given by 
$$\Phi_{mn}(x) = \prod_{i=1}^{\phi(n)/d_n}
\left(\prod_{d \mid   m}\Psi_{i,d}\left(x^d\right)^{\mu(m/d)}\right)$$
where each $\Psi_{i,d}$ is the minimal polynomial of
$\xi_{n,i}^d$ with $\xi_{n,i}$ a root of $f_i.$
\end{thm}

\begin{proof}
Since $q$ is a primitive root modulo $m,$ $\Phi_m$ is irreducible over $\F.$
Note $\gcd(d_n,\phi(m)) = 1$ implies each polynomial $f_i\odot
\Phi_m$ is irreducible over $\F$ by Theorem \ref{thm 1}. Then by Corollary
\ref{Phi_mn} and Theorem \ref{thm 3} the complete factorization of
$\Phi_{mn}$ over $\F$ is given by $$\Phi_{mn}(x)  =
\prod_{i=1}^{\phi(n)/d_n}\left(f_i\odot \Phi_m\right)(x) = \prod_{i=1}^{\phi(n)/d_n}\left(\prod_{d \mid   m}\Psi_{i,d}\left(x^d\right)^{\mu(m/d)}\right)
$$
as required.
\end{proof}

\begin{rem}
Note that the irreducible factors of $\Phi_{mn}$ are
expressed in terms of the minimal polynomials $\Psi_{i,d}$ over $\F$ of
$\xi_{n,i}^d,$ where the root $\xi_{n,i}$ of $f_i$ is a primitive $n$-th root
of unity. We remark that it is not necessary to compute the minimal polynomials: Since
$\gcd(m,n) = 1,$ then $\gcd(d,n) = 1$ for each $d \mid   m;$ hence $\xi_{n,i}^d$ is 
a primitive $n$-th root of unity, and so it must be a root of some irreducible 
factor $f_j$ of $\Phi_n.$ But then $\Psi_{i,d} = f_j.$ 

As a particular consequence, we can now let $\Phi_n$ be as in
Theorems \ref{2^n and q = 1 mod 4}, \ref{2^nr and q = 1 mod 4}, \ref{2^n and q
= 3 mod 4}, \ref{2^nr and q = 3 mod 4}, etc, and then use the respective
factorizations $\prod_i f_i$ given there to factor $\Phi_{mn}.$ This is now merely a matter of computation.

On the other hand, in the case that we do not know the factorization of
$\Phi_n,$ we can let $S = \{\xi_{n_i}\}_{i=1}^{\phi(n)/{d_n}}$ be a set of
pairwise non-conjugate primitive $n$-th roots of unity $\xi_n.$ Then we can
write the complete factorization of $\Phi_{mn}$ over $\F$ as
$$\Phi_{mn}(x) =
\prod_{i=1}^{\phi(n)/d_n}
\left(\prod_{d \mid   m}\Psi_{i,d}\left(x^d\right)^{\mu(m/d)}\right) =
\prod_{\xi_{n_i}\in S}
\left(\prod_{d \mid   m}\Psi_{i,d}\left(x^d\right)^{\mu(m/d)}\right)$$
where $\Psi_{i,d}$ is the minimal polynomial of $\xi_{n_i}^d.$  
Indeed, $\xi_{n_i}$ is a root of $\Psi_{i,1} = f_i,$ and for
non-conjugates $\xi_{n_i},\ \xi_{n_j},$ we have $f_i \neq f_j;$
finally, there are $ |S| = \phi(n)/d_n$ irreducible factors $f_i$ of
$\Phi_n.$
\end{rem}

\begin{lem}[\bf{Theorem 3.35, \cite{Lidl}}]\label{composition}
Let $f_1,\ f_2,\dots,\ f_N$ be all distinct monic irreducible
polynomials in $\F[x]$ of degree $m$ and order $e,$ and let $t\geq 2$ be an
integer whose prime factors divide $e$ but not $\left(q^{m}-1\right)/e.$
Assume also that $q^m \equiv 1 \pmod{4}$ if $t \equiv 0 \pmod{4}.$ Then
$f_1\left(x^t\right),\ f_2\left(x^t\right),\dots,\ f_N\left(x^t\right)$ are all
distinct monic irreducible polynomials in $\F[x]$ of degree $mt$ and order $et.$
\end{lem}

\begin{lem}[\bf{Exercise 2.57, \cite{Lidl}}]\label{tricks}{\ }\\
 (a) $\Phi_{2n}(x) = \Phi_n(-x)$ for $n\geq 3$ and $n$ odd.\\ 
 (b) $\Phi_{mt}(x) = \Phi_m\left(x^t\right)$ for all positive integers $m$
 that are divisible by the prime $t.$\\ 
 (c) $\Phi_{mt^k}(x) = \Phi_{mt}\left(x^{t^{k-1}}\right)$ if $t$ is a prime
 and $m,\ k$ are arbitrary positive integers.
 \end{lem}
 
 Note that Lemma \ref{tricks} implies that, in particular, for $n\geq 2,$
$\Phi_{2^nr}(x) = \Phi_{2^{n-1}r}(x^2).$ Observe that if $\Phi_{2^{n-1}r} =
\prod_i h_i$ is the corresponding factorization, then $\Phi_{2^nr}(x) = \Phi_{2^{n-1}r}\left(x^2\right) = \prod_i
h_i\left(x^2\right).$ \emph{This means that we can obtain all the irreducible
factors of $\Phi_{2^{n}r}$ by factoring each $h_i\left(x^2\right).$}

Let $v_2(k)$ denote the highest power of $2$ dividing $k.$

\begin{lem}[\bf{Proposition 1, \cite{Beyl}}]\label{v2} For $i\geq 1,$
\begin{eqnarray*}
v_2\left(q^i - 1\right) &=& v_2(q-1) + v_2\left(q^{i-1} + q^{i-2}
+\dots + 1\right)\\
&=& \begin{cases} v_2(q-1) + v_2(i) + v_2(q+1) - 1, & \mbox{if } i\mbox{ is
even} \\ v_2(q-1), &
\mbox{if } i\mbox{ is odd.} \end{cases}
\end{eqnarray*}
\end{lem}

\begin{lem}\label{lem 7}
Let $q = p^s$ be a power of an odd prime $p,$ let $r\geq 3$ be any odd number
coprime to $q,$ and let $d_r = \ord_r(q).$ If $q \equiv 1\pmod{4},$ write $q = 2^Am + 1,\
A\geq 2,\ m$ odd. Otherwise if $q \equiv 3\pmod{4},$ write $q = 2^Am - 1,\ A\geq 2,\ m$ odd. Set $K:=v_2\left(q^{d_r} - 1\right).$ Then if $d_r$ is even, in both cases
cases of $q$ we have $K = A + v_2(d_r) > A \geq 2.$ If $d_r$ is odd and
$q\equiv 1 \pmod{4},$ then $K = A.$ If $d_r$ is odd and $q\equiv 3 \pmod{4},$ then
$K = 1.$
\end{lem}

\begin{proof}
First assume $d_r$ is even. Then $v_2\left(d_r\right) > 0,$ and so $A + v_2(d_r)
> A \geq 2.$ If $q\equiv 1\pmod{4},$ we have $q - 1 = 2^Am$ and $q+1 = 2\left(2^{A-1}m + 1\right) = 2m',$ where $m'$ is odd. Thus $v_2(q-1) = A,$ and $v_2(q+1) = 1.$ 
Hence, $K = v_2(q-1) + v_2(d_r) + v_2(q+1) - 1 = A + v_2(d_r).$

If $q\equiv 3\pmod{4},$ we have $q - 1 = 2\left(2^{A-1}m - 1\right)$ and $q+1
= 2^Am.$ Thus $v_2(q-1) = 1$ and $v_2(q+1) = A.$ Hence, $K = v_2(q-1) +
v_2(d_r) + v_2(q+1) - 1 = A + v_2(d_r).$

Now if $d_r$ is odd, by Lemma \ref{v2}, $K = v_2(q-1).$ If $q\equiv 1 \pmod{4},$
then $K = A.$ Otherwise, if $q\equiv 3 \pmod{4},$ then $K = 1.$
\end{proof}
 
The following result represents an improvement over Theorem \ref{L} in \cite{Prof}. 
Later on we use it often in the following sections.
 
\begin{thm}\label{thm 5}
Let $q = p^s$ be a power of an odd prime $p,$ let $r\geq 3$ be
any odd number such that $\gcd(r,q) = 1.$ Let $d_r = \ord_r(q).$ If $d_r$ is
odd, further assume $q \equiv 1 \pmod{4}.$ Set $K:= v_2\left(q^{d_r}-1\right).$ Then for
$n\leq K$ and any irreducible factor $h_n$ of $\Phi_{2^nr},$ we
have $\deg (h_n) = d_r.$ Furthermore, if $0 < n < K$ strictly, then
$h_n\left(x^2\right)$ decomposes into precisely two irreducible factors of degree $d_r$ which are irreducible factors of $\Phi_{2^{n+1}r}.$ On the other hand, for $n>K,$ and any irreducible factor
$h_K$ of $\Phi_{2^Kr}$ over $\F,$ $h_K\left(x^{2^{n-K}}\right)$ is also irreducible over $\F.$ Moreover, all irreducible factors of $\Phi_{2^nr}$ are obtained in this way.
\end{thm}

\begin{proof}
Since $q^{d_r} \equiv 1 \pmod{r}$ and $K = v_2\left(q^{d_r} -
1\right),$ we have $q^{d_r} \equiv 1 \pmod{2^Kr}.$ Let $n\leq K.$ It is
true that $q^{d_r} \equiv 1 \pmod{2^nr}.$ Let $d_n = \ord_{2^nr}(q).$ Then
$d_n\mid   d_r.$ On the other hand, $q^{d_n} \equiv 1 \pmod{2^nr}$ gives $q^{d_n}
\equiv 1 \pmod{r}$ implying $d_r \mid   d_n.$ Consequently, $d_n = d_r.$ Recalling
that the degree of each irreducible factor of $\Phi_{2^nr}$ is $\ord_{2^nr}(q) = d_n,$ we conclude that for $n\leq K,$ each
irreducible factor of $\Phi_{2^nr}$ has degree $d_r.$

For $0< n < K,$ let $h_n$ be an irreducible factor (of degree $d_r$) of
$\Phi_{2^nr}.$ Then $h_n\left(x^2\right)$ has degree $2d_r$ and is a factor
of $\Phi_{2^{n+1}r}$ clearly. Because $n+1\leq K,$ then $h_n\left(x^2\right)$
must decompose into an amount $z$ of irreducibles of degree $d_r.$
But this is possible only if $z = 2.$

Note $e = 2^Kr$ is the order of $\Phi_{2^{K}r}$ and thus the order of any
irreducible factor $h_K$ of it. By definition, $2^{K+1}\nmid \left(q^{d_r} - 1\right).$
Hence, $2\nmid \left(q^{d_r} - 1\right)/e,$ and by Lemma \ref{composition},
$h_K\left(x^2\right)$ is irreducible over $\F.$ If $d_r$ is even, then $K >
2$ by Lemma \ref{lem 7}. If $d_r$ is odd, then $q\equiv 1 \pmod{4}$ by assumption,
and so $K = A \geq 2$ by Lemma \ref{lem 7}. Then $2^2 = 4\mid  \left(q^{d_r}-1\right).$ As a result, for
$n>K,$ Lemma \ref{composition} gives $h_K\left(x^{2^{n-K}}\right)$ is
irreducible over $\F.$ Because $$\Phi_{2^nr}(x) = \Phi_{2^Kr}\left(x^{2^{n-K}}\right) = \prod_i h_{K_i}\left(x^{2^{n-K}}\right),$$ where $\Phi_{2^Kr} = \prod_i h_{K_i}$ is the corresponding factorization,
the factorization of $\Phi_{2^nr}$ over $\F$ is complete. Thus we can obtain
all irreducible factors of $\Phi_{2^nr}$ in this way.
\end{proof}

Whenever $d_r$ is even, or $q \equiv 1 \pmod{4},$ the bound $K =
v_2\left(q^{d_r}-1\right)$ in Theorem \ref{thm 5} represents an improvement over
the bound $L = v_2\left(q^{\phi(r)}-1\right)$ of Theorem \ref{L} due to L. Wang
and Q. Wang \cite{Prof}. This is because $K \leq L$ as
$\left(q^{d_r}-1\right)\mid   \left(q^{\phi(r)}-1\right).$ Moreover, it is clear
that $K$ is the smallest bound with the property that $\Phi_{2^nr}(x) = \prod_i
h_{K_i}\left(x^{2^{n-K}}\right)$ is the corresponding factorization over $\F$
for $n > K.$ In
Theorem \ref{2^nr and q = 3 mod 4} we will show that, in particular, when $d_r$
is odd and $q \equiv 3 \pmod{4},$ the corresponding bound is $v_2(q+1) = A.$ That is, if $\Phi_{2^Ar} = \prod_i h_{A_i}$ is the corresponding factorization, then for $n > A$ the factorization of $\Phi_{2^nr}$ over $\F$ is given by
$\Phi_{2^nr}(x) = \prod_i h_{A_i}\left(x^{2^{n-A}}\right).$

Before we move on to the following sections we need the following notations. Let
$\Omega(r)$ be the set of $r$-th primitive roots of unity and let $U_n$ be the set of the $2^n$-th primitive roots of unity. 
Similarly as done in  \cite{Prof} we let the expression $$\prod_{a\in A}\dots \prod_{b\in
B}f_i(x,a,\dots,b)$$ denote the product of \emph{distinct} polynomials
$f_i(x,a,\dots,b)$ satisfying conditions $a\in A,\dots,b\in B.$ For example, if
we let $g_w$ be an irreducible factor of $\Phi_r$ with root $w,$ say in $\mathbb{F}_{q^{d_r}},$ then in the product $\prod_{w\in \Omega(r)} g_w$ we take $g_w$ and not any of
$g_{w^{q^i}}$ as $g_w = g_{w^{q^i}}$ in this case.

Recall the \emph{elementary symmetric polynomials} $S_i$ defined by
$$S_i(x_1,x_2,\dots,x_n) = \sum_{k_1<k_2<\dots <k_i} x_{k_1}x_{k_2}\dots x_{k_i} 
$$
for any $i=1,2,\dots,n,$ with $S_0 = 1.$ The following proposition is a well
known fact.

\begin{prop}[{\bf Theorem 3, Section 4.5, \cite{Nicholson}}]
Write $S_i = S_i(x_1,x_2,\dots,x_n)$ for $1\leq i\leq n.$ Then 
$$\prod_{i=1}^n \left(x-x_i\right) = \sum_{i=0}^n\left(-1\right)^iS_ix^{n-i}.
$$
\end{prop}

From now on for any proper element $w\in \Fn,$ i.e. $\F(w) = \Fn,$ we use the
notation $S_{i,w} = S_i\left(w,w^q,\dots,w^{q^{n-1}}\right).$

\subsection{Factorization of $\Phi_{2^nr}$ when $q \equiv 1 \pmod{4}$} 
In this section and the following we make the assumption that the explicit
factorization of $\Phi_r$ is given to us as a known. One may use for instance the results
due to Stein (2001) to compute the factors of $\Phi_r$ efficiently when $q
= p$ and $r$ is an odd prime distinct to $p.$ First, we need the following
well known theorem concerning the factorization of $\Phi_{2^n}$ when $q \equiv 1 \pmod{4}$ which follows from Theorems 2.47 and
3.35 in \cite{Lidl}. 

\begin{thm}[{\bf \cite{Lidl}}]\label{2^n and q = 1 mod 4}
Let $q \equiv 1 \pmod{4},$ i.e. $q = 2^Am+1,\ A \geq 2,\ m$ odd. Let $U_n$ denote
the set of primitive $2^n$-th roots of unity.

(a) If $1\leq n\leq A,$ then $\ord_{2^n}(q) = 1$ and $\Phi_{2^n}$ is the
product of $2^{n-1}$ irreducible linear factors over $\F:$
$$\Phi_{2^n}(x) = \prod_{u\in U_n}\left(x+u\right).$$

(b) If $n>A,$ then $\ord_{2^n}(q) = 2^{n-A}$ and $\Phi_{2^n}$ is the product
of $2^{A-1}$ irreducible binomials over $\F$ of degree $2^{n-A}:$
$$\Phi_{2^n}(x) = \prod_{u\in U_A}\left(x^{2^{n-A}} + u\right).$$
\end{thm}

First recall that whenever $\gcd(q,n) =
1$, $\Phi_n$ decomposes into $\phi(n)/\ord_n(q)$ irreducibles over $\F$ of
degree $\ord_n(q)$ (Theorem 2.47, \cite{Lidl}). In particular, $\Phi_r$
decomposes into irreducibles of degree $d_r = \ord_r(q)$ over $\F$ when $q,\ r$ are coprime.

We now give the factorization of $\Phi_{2^nr}$ when
$q \equiv 1 \pmod{4}.$

\begin{thm}\label{2^nr and q = 1 mod 4}
Let $q \equiv 1 \pmod{4},$ say $q = 2^Am+1,\ A \geq 2,\ m$ odd. Let $r \geq 3$
be odd such that $\gcd(q,r) = 1,$ and let $d_r
= \ord_r(q).$
\\
1. If $1\leq n\leq A,$ then 
$$\Phi_{2^nr}(x) = \prod_{u\in U_n}\prod_{w\in \Omega(r)}\left(\sum_{i=0}^{d_r}
u^i S_{i,w}x^{d_r-i}\right)$$
is the complete factorization of $\Phi_{2^nr}$ over $\F.$
\\
2. If $n>A,$ we have:
\\
(a) If $d_r$ is odd, then
$$\Phi_{2^nr}(x) = \prod_{u\in U_A}\prod_{w\in \Omega(r)}\left(\sum_{i=0}^{d_r}
u^i S_{i,w}x^{2^{n-A}\left(d_r-i\right)}\right). $$
is the complete factorization of $\Phi_{2^nr},\ n>A,$ over $\F.$
\\
(b) If $d_r$ is even, then:

(i) For $A < n \leq K,$ the complete factorization of $\Phi_{2^nr}$ over
$\F$ is given by
$$\Phi_{2^nr}(x) =
\prod_{u\in
U_A}\prod_{w\in
\Omega(r)}\left(x^{d_r}+
\sum_{i=1}^{d_r}a_{n_i}x^{d_r-i}\right)$$ where each $a_{n_i},\ 1 \leq i\leq
d_r,$ satisfies the following system of non-linear recurrence relations 
$$\Bigg\{\sum_{i+j=2k}(-1)^ja_{n_i}a_{n_j} = a_{(n-1)_k}, \tab 1\leq k\leq d_r \Bigg\}$$
with initial values $a_{A_k} = u^kS_{k,w},\ 1\leq k\leq d_r.$

(ii) For $n > K,$ the complete factorization of $\Phi_{2^nr}$ over $\F$ is
given by $$\Phi_{2^nr}(x) = \prod_{u\in U_A}\prod_{w\in
\Omega(r)}\left(x^{2^{n-K}d_r}+
\sum_{i=1}^{d_r}a_{K_i}x^{2^{n-K}(d_r-i)}\right)$$ 
where each $a_{K_i},\ 1\leq i\leq d_r,$ is as obtained in
(i).
\end{thm}

\begin{proof}
Let $$\Phi_r(x) = \prod_{w\in \Omega(r)}g_w(x) = \prod_{w\in
\Omega(r)}\left(\sum_{i=0}^{d_r}(-1)^i S_{i,w}x^{d_r-i}\right)$$
be the factorization of $\Phi_r$ over $\F.$

1. By Theorem \ref{2^n and q = 1 mod 4} (a) and Corollary \ref{Phi_mn} we have
$$\Phi_{2^nr}(x) = \left(\Phi_{2^n}\odot \Phi_r\right)(x) = \prod_{u\in
U_n}\prod_{w\in \Omega(r)}\left((x+u)\odot g_w\right)(x).$$
By Proposition \ref{comput},
\begin{eqnarray*}
\left((x+u)\odot g_w\right)(x) &=&
(-u)^{d_r}g_w\left((-u)^{-1}x\right) =
(-u)^{d_r}\sum_{i=0}^{d_r}(-1)^iS_{i,w}(-u)^{i-d_r}x^{d_r-i}\\
&=&
\sum_{i=0}^{d_r}S_{i,w}u^ix^{d_r-i}.
\end{eqnarray*}
Noting that each $(x+u) \odot g_w$ is irreducible over $\F$ by
Theorem \ref{thm 1}, these factors give us a complete factorization of
$\Phi_{2^nr}$ over $\F$ for $1 \leq n\leq A.$
\\\\
2 (a): Since $q \equiv 1 \pmod{4}$ and $d_r$ is odd, Lemma \ref{lem 7} gives $K =
A;$ consequently if $\Phi_{2^Ar} = \prod_i h_{A_i}$ is the corresponding
factorization over $\F$, then Theorem \ref{thm 5} gives that for $n > A,$ the
complete factorization of $\Phi_{2^nr}$ over $\F$ is given by $\Phi_{2^nr}(x) =
\prod_i h_{A_i}\left(x^{2^{n-A}}\right).$ Thus it only remains to make the substitution $x \rightarrow
x^{2^{n-A}}$ in each irreducible factor $h_{A_i}$ obtained in Part
1 as the statement in the theorem shows.

(b) (i) ($A<n \leq K$ and $d_r$ even): Let $h_{n-1}$ be an irreducible factor
of $\Phi_{2^{n-1}r}.$ By Theorem \ref{thm 5}, $\deg (h_{n-1}) = d_r$ and
$h_{n-1}\left(x^2\right)$ decomposes into two irreducibles of degree $d_r$ which
are irreducible factors of $\Phi_{2^nr}.$ Let $h_{n-1}\left(x^2\right) =
f_n(x) g_n(x)$ be the corresponding factorization. First, we show $g_n(x) =
f_n(-x).$ Let $\alpha$ be a root of $f_n.$ We claim that $-\alpha$ is not a
root of $f_n.$ On the contrary, suppose $f_n(-\alpha) = 0.$ Then $-\alpha =
\alpha^{q^i}$ for some $i \in [0,d_r-1]$ implies $-1 = \alpha^{q^i-1}$ and
$1 = \alpha^{2\left(q^i-1\right)}.$ But then $\ord(\alpha) = 2^nr \mid  
2\left(q^i-1\right)$ and so $r \mid   \left(q^i-1\right).$ However, this
contradicts $\ord_r(q) = d_r > i.$ Therefore $f_n(-\alpha) \neq 0.$ Now, we have
$$
f_n(-\alpha)g_n(-\alpha) = h_{n-1}\left((-\alpha)^2 \right) =
h_{n-1}\left( \alpha^2 \right) = f_n(\alpha) g_n(\alpha) = 0.
$$
As $f_n(-\alpha) \neq 0,$ necessarily $g_n(-\alpha) = 0.$ Thus both
$f_n(-x),\ g_n(x)$ have $-\alpha$ as a root. But then since both $f_n(-x),\
g_n(x)$ are monic irreducible polynomials over $\F$ of degree $d_r,$ it must be
that $g_n(x) = f_n(-x).$ Therefore $h_{n-1}(x^2) = f_n(x)f_n(-x)$ is the
corresponding factorization. We may write
$$
h_{n-1}(x) = x^{d_r}+ \sum_{k=1}^{d_r}a_{(n-1)_k}x^{d_r-k}
$$
and 
$$
f_n(x) = x^{d_r} + \sum_{i=1}^{d_r}a_{n_i}x^{d_r - i}
$$
for some coefficients $a_{(n-1)_k},\ a_{n_i} \in \F.$ Now, $h_{n-1}(x^2) = f_n(x)f_n(-x)$ gives
\begin{eqnarray*}
x^{2d_r}+ \sum_{k=1}^{d_r}a_{(n-1)_k}x^{2(d_r-k)}
&=&\left(x^{d_r}+\sum_{i=1}^{d_r}a_{n_i}x^{d_r-i}\right)\left(x^{d_r}+\sum_{j=1}^{d_r}a_{n_j}(-1)^jx^{d_r-j}\right)\\
&=&
x^{2d_r} + \sum_{k=1}^{2d_r}\sum_{i+j=k}(-1)^ja_{n_i}a_{n_j}x^{2d_r-k}\\
&=& 
x^{2d_r} + \sum_{k=1}^{d_r}\sum_{i+j=2k}(-1)^ja_{n_i}a_{n_j}x^{2(d_r-k)}.
\end{eqnarray*}
The last equality followed from the fact that the coefficients of odd powers of
$x$ in $h_{n-1}(x^2)$ are $0.$ Comparing coefficients on each side we see that each $a_{n_i},\ 1\leq i\leq
d_r,$ satisfies the following system of non-linear equations
$$\Bigg\{\sum_{i+j=2k}(-1)^ja_{n_i}a_{n_j} = a_{(n-1)_k},\tab 1\leq k\leq d_r
\Bigg\}.$$
We know the system must have a solution, otherwise $h_{n-1}(x^2) \neq
f_n(x)f_n(-x)$ contrary to the previous arguments. Moreover, the solution must be unique
by the uniqueness of factorizations. Furthermore, the reader can see that we can
obtain the coefficients of $f_n,$ and hence of $f_n(-x),$ 
by a recursion where the initial values are the coefficients 
$a_{A_k} = u^kS_{k,w},\ 1\leq k\leq d_r$ of an irreducible factor of $\Phi_{2^Ar}$ 
which we already know from Part 1. Next, we show that we can obtain all the
irreducible factors of $\Phi_{2^nr}$ in this way. We claim that for any two
distinct initial-value sets $I = \{u_i^kS_{k,w} \}, \ J = \{u_j^kS_{k,w} \},$
all the irreducible factors generated by $I$ and $J$ are distinct. By induction
on $n$ where $A < n \leq K$: Let $g_A,\ h_A$ be the distinct 
irreducible factors of $\Phi_{2^Ar}$ corresponding to $I$ and $J.$ Then in
particular $g_A(x^2) \neq h_A(x^2).$ As each of these decomposes into two irreducible factors of the form $f_{A+1}(x),\ f_{A+1}(-x),$ then all four
irreducible factors must be distinct. Otherwise if they share an irreducible factor, say
$f_{A+1}(-x),$ then necessarily they must share $f_{A+1}(x)$ resulting in
$g_A(x^2) = h_A(x^2),$ a contradiction. Similarly one can show that the
inductive step follows from the inductive hypothesis. The claim now
follows. Consequently, if we let $s = n - A,$ then each initial-value set
$\{u^kS_{k,w}\}$ corresponding to an irreducible factor $g_A$ of $\Phi_{2^Ar}$ will generate a total of $2^s$ distinct irreducible factors of $\Phi_{2^nr}.$ Since there are
$\phi(2^Ar)/d_r$ irreducible factors of $\Phi_{2^Ar},$ the initial-value
sets generate a total of $2^s\phi(2^Ar)/d_r = 2^{s+A-1}\phi(r)/d_r = 2^{n-1}\phi(r)/d_r =
\phi(2^nr)/d_r$ distinct irreducible factors of $\Phi_{2^nr},$ as desired.
The factorization is complete.

(ii) ($n > K$): If $\Phi_{2^Kr} = \prod_i h_{K_i}$ is the corresponding
factorization, then by Theorem \ref{thm 5}, for $n > K,$ we obtain
$\Phi_{2^nr}(x) = \prod_i h_{K_i}\left(x^{2^{n-K}}\right)$ as its complete factorization. Since
each $h_{K_i}$ is already known from Part (i), it only remains to make the substitution
$x\rightarrow x^{2^{n-K}}$ in each $h_{K_i}$ to obtain each irreducible
factor of $\Phi_{2^nr},$ as the statement in the theorem shows. The proof of
(ii) is complete.
\end{proof}

\begin{rem}\label{rem 2.1}
In order to obtain each irreducible factor of $\Phi_{2^nr},$
for any $n\in \mathbb{N},$ we require at most $v_2(d_r)$ iterations
of the system of non-linear recurrence relations in (i): For $n\leq A,$
the explicit factorization is already given in Part 1. However, for $A <n \leq
K$ and $d_r$ even, the system of non-linear recurrence relations in (i) must
iterate for $n-A$ steps. In the case $A< n = K,$ the system will iterate for the
maximum number of steps $K-A.$ By Lemma \ref{lem 7}, this equals $v_2(d_r).$
\end{rem}

\begin{rem}\label{rem 2.2}
We can also formulate the factorization of $\Phi_{2^nr},\ 1\leq 
n \leq K,$ in terms of the non-linear recurrence relation in (i) with initial
values corresponding to $n = 1.$ For small finite fields and small $d_r,$ this
can be computed fairly fast.
\end{rem}

\begin{rem}\label{rem 2.3}
Let $n > K,$ let $S = \{s_k\},\ T = \{t_k\}$ be homogeneous LRS's with
characteristic polynomials $\Phi_{2^n},\ \Phi_{r}$ respectively. Then as discussed earlier, 
the characteristic polynomial of $ST = \{s_k t_k \}$ is
$\Phi_{2^nr} = \Phi_{2^n} \odot \Phi_r.$ Since all irreducible factors
of $\Phi_{2^nr},\ n > K,$ have degree $2^{n-K}d_r,$ the minimal polynomial of
$ST$ must have degree $2^{n-K}d_r.$ This is the linear complexity of $ST.$ Note
that if we let $n \rightarrow \infty,$ the linear complexity of the
corresponding LRS $ST$ approaches infinity.
\end{rem}

For the subcases $q \equiv 1 \pmod{4}$ with $q \equiv \pm 1 \pmod{r}$ and thus $d_r
= 1,\ 2,$ where $r$ is an odd prime, Theorem \ref{2^nr and q = 1 mod 4} becomes Theorem 1, Parts 2 and
4 in Fitzgerald and Yucas (2007) \cite{Fitzgerald 2007}.

\subsection{Factorization of $\Phi_{2^nr}$ when $q \equiv 3 \pmod{4}$}

We need the following result due to Meyn (1996) \cite{Meyn}.

\begin{thm}[\bf{Theorem 1, \cite{Meyn}}]\label{2^n and q = 3 mod 4}
Let $q \equiv 3 \pmod{4},$ i.e. $q = 2^Am - 1,\ A\geq 2,\ m$ odd. Let $n\geq 2.$

(a) If $n\leq A,$ then $\Phi_{2^n}$ is the product of $2^{n-2}$
irreducible trinomials over $\F:$
$$\Phi_{2^n}(x) = \prod_{u\in U_n}\left(x^2 + \left(u + u^{-1}\right)x +
1\right).$$

(b) If $n > A,$ then $\Phi_{2^n}$ is the product of $2^{A-2}$ irreducible
trinomials over $\F:$
$$\Phi_{2^n}(x) = \prod_{u\in
U_A}\left(x^{2^{n-A+1}}+\left(u-u^{-1}\right)x^{2^{n-A}}-1\right). $$
\end{thm}

We are now ready to give the factorization of $\Phi_{2^nr}$ when $q \equiv 3
\pmod{4}.$

\begin{thm}\label{2^nr and q = 3 mod 4}
Let $q \equiv 3 \pmod{4},$ i.e. $q = 2^Am - 1,\ A\geq 2,\ m$ odd. Let $r \geq 3$
be odd such that $\gcd(q,r) = 1,$ and let $d_r
= \ord_r(q).$
\\
1. If $n = 1,$ then 
$$\Phi_{2r}(x) = \prod_{w\in\Omega(r)}\left(\sum_{i=0}^{d_r}S_{i,w}x^{d_r -
i}\right) $$
is the complete factorization of $\Phi_{2r}$ over $\F.$
\\
2. If $2\leq n\leq A,$ we have:

(i) If $d_r$ is odd, the complete factorization of $\Phi_{2^nr}$ over $\F$ is
given by $$\Phi_{2^nr}(x) = \prod_{u\in U_n}\prod_{w\in
\Omega(r)}\left(\sum_{k=0}^{2d_r}\sum_{i+j=k}S_{i,w}S_{j,w}u^{i-j}x^{2d_r-k}
\right).$$

(ii) If $d_r$ is even, $\Phi_{2^nr}$ decomposes into irreducibles of degree
$d_r$ over $\F$ so that
$$\Phi_{2^nr}(x) = \prod_{u\in U_n}\prod_{w\in
\Omega(r)}\Bigg[\left(x^{d_r}+ \sum_{i=1}^{d_r}a_{n_i}x^{d_r-i} \right) \left(x^{d_r}+
\sum_{j=1}^{d_r}b_{n_j}x^{d_r-j} \right)\Bigg]$$ 
is the complete factorization of $\Phi_{2^nr}$ over $\F,$ where each
$a_{n_i},\ b_{n_j} \in \F,\ 1 \leq i,\ j \leq d_r,$ satisfies the following system of equations $$\Bigg\{\sum_{i+j = k}a_{n_i}b_{n_j} = \sum_{i+j = k}
S_{i,w}S_{j,w}u^{i-j}, \tab   1\leq k\leq 2d_r
\Bigg\}.$$
\\
3. If $d_r$ is odd, then for $n> A$ the complete factorization of
$\Phi_{2^nr}$ over $\F$ is given by
$$\Phi_{2^nr}(x) = \prod_{u\in U_A}\prod_{w\in
\Omega(r)}\left(\sum_{k=0}^{d_r}\sum_{i+j=k}u^{i-j}S_{i,w}S_{j,w}x^{2^{n-A}(2d_r-k)}\right).$$
\\
4. If $d_r$ is even, we have:

(iii) For $A < n \leq K,$ the complete factorization of $\Phi_{2^nr}$ over
$\F$ is given by
$$\Phi_{2^nr}(x) =
\prod_{u\in
U_A}\prod_{w\in
\Omega(r)}\left(x^{d_r}+
\sum_{i=1}^{d_r}a_{n_i}x^{d_r-i}\right)$$ where each $a_{n_i},\ 1 \leq i\leq
d_r,$ satisfies the following system of non-linear recurrence relations 
$$\Bigg\{\sum_{i+j=2k}(-1)^ja_{n_i}a_{n_j} = a_{(n-1)_k}, \tab          
1\leq k\leq d_r \Bigg\}$$
with initial values $a_{A_k},\ 1\leq k\leq d_r,$ as obtained in (ii).

(iv) For $n > K,$ the complete factorization of $\Phi_{2^nr}$ over $\F$ is
given by $$\Phi_{2^nr}(x) = \prod_{u\in U_A}\prod_{w\in
\Omega(r)}\left(x^{2^{n-K}d_r}+
\sum_{i=1}^{d_r}a_{K_i}x^{2^{n-K}(d_r-i)}\right)$$ 
where each $a_{K_i},\ 1\leq i\leq d_r,$ is as obtained in
(iii).
\end{thm}

\begin{proof}
Let $$\Phi_r(x) = \prod_{w\in \Omega(r)}g_w(x) = \prod_{w\in
\Omega(r)}\left(\sum_{i=0}^{d_r}(-1)^i S_{i,w}x^{d_r - i}\right)$$
be the factorization of $\Phi_r$ over $\F.$
\\
1. $(n=1):$ Because $g_w$ is irreducible over $\F$, $g_w(-x)$ is irreducible
over $\F.$ By Theorem \ref{tricks},
\begin{eqnarray*}
\Phi_{2r}(x) &=& \Phi_r(-x) = \prod_{w\in \Omega(r)}g_w(-x) =
\prod_{w\in \Omega(r)}\left(\sum_{i=0}^{d_r}(-1)^{d_r} S_{i,w}x^{d_r -
i}\right).
\end{eqnarray*}
Note that in the case $d_r$ is odd the number of irreducible factors of
$\Phi_{2r},$ which is $\phi(r)/d_r,$ is even. Thus, it follows
that we may write the factorization above as $$\Phi_{2r}(x) = \prod_{w\in
\Omega(r)}\left(\sum_{i=0}^{d_r}S_{i,w}x^{d_r - i}\right). 
$$
The factorization is complete.
\\
2. $\left(2\leq n\leq A\right):$ By Theorem \ref{2^n and q = 3 mod 4} (a) we
have
\begin{eqnarray*}
\Phi_{2^nr}(x) &=& \prod_{u\in U_n}\prod_{w\in
\Omega(r)}\left(\left(x^2 + \left(u + u^{-1}\right)x +
1\right) \odot g_w\right)(x)\\
 &=& 
\prod_{u\in U_n}\prod_{w\in
\Omega(r)}(-u)^{d_r}g\left((-u)^{-1}x\right)(-u)^{-d_r}g\left(-ux\right)\\ 
&=& \prod_{u\in U_n}\prod_{w\in
\Omega(r)}\left(\sum_{i=0}^{d_r}(-1)^iS_{i,w}(-u)^{i-d_r}x^{d_r-i}\right)\left(\sum_{j=0}^{d_r}(-1)^jS_{j,w}(-u)^{d_r-j}x^{d_r-j}\right)\\
&=&
\prod_{u\in U_n}\prod_{w\in
\Omega(r)}\left(\sum_{k=0}^{2d_r}\sum_{i+j =
k}S_{i,w}S_{j,w}u^{i-j}x^{2d_r-k}\right). \tab \tab (*)\\  
\end{eqnarray*}
First, note that these factors in $(*)$ are over $\F$ as the composed product of
polynomials over $\F$ are polynomials over $\F.$ We have:

(i) If $d_r$ is odd, then $\gcd(2,d_r) = 1$ and so each factor $\left(x^2
+ \left(u + u^{-1}\right)x + 1\right) \odot g_w$ is irreducible by
Theorem \ref{thm 1}; hence the factorization is complete. 

(ii) If $d_r$ is even, then in particular $A < A + v_2(d_r) = K.$ Then by
Theorem \ref{thm 5} each factor in $(*)$ of $\Phi_{2^nr}$
must decompose into two irreducibles of degree $d_r.$ Thus, for some
coefficients $a_{n_i},\ b_{n_j} \in \F$ we must have
\begin{eqnarray*} 
\sum_{k=0}^{d_r}\sum_{i+j = k}S_{i,w}S_{j,w}u^{i-j}x^{2d_r-k} 
&=&
\left(x^{d_r}+ \sum_{i=1}^{d_r}a_{n_i}x^{d_r-i} \right) \left(x^{d_r}+
\sum_{j=1}^{d_r}b_{n_j}x^{d_r-j} \right)\\
&=&
x^{2d_r} + \sum_{k=1}^{2d_r}\sum_{i+j=k}a_{n_i}b_{n_j}x^{2d_r-k}.
\end{eqnarray*}
Comparing coefficients on each side we see that each $a_{n_i},\ b_{n_j}, \ 1\leq
i,\ j\leq d_r,$ satisfies the following system of equations
$$\Bigg\{\sum_{i+j = k}a_{n_i}b_{n_j} = \sum_{i+j =
k}S_{i,w}S_{j,w}u^{i-j},\tab 1\leq k\leq 2d_r \Bigg\} $$
which has a solution. We stress that the solution must be unique by the
uniqueness of factorizations. Hence the result follows.
\\\\
3. ($n > A$ and $d_r$ odd): Since $\gcd\left(2^{n-A+1},d_r\right) = 1,$ the
complete factorization of $\Phi_{2^nr}$ over $\F$ is given by
$$
\Phi_{2^nr}(x) = \prod_{u\in U_A}\prod_{w\in
\Omega(r)}\left(\left(x^{2^{n-A+1}}+\left(u-u^{-1}\right)x^{2^{n-A}}-1\right)
\odot g_w\right)(x).$$ 
Since the computation
of the composed product above is somewhat more involved this time, we proceed as
follows: First note that for $n > A$ all irreducible factors of $\Phi_{2^nr}$
have degree $2^{n-A+1}d_r.$ It then follows that if a factor of $\Phi_{2^nr}$ has
degree $2^{n-A+1}d_r,$ it must be an irreducible factor. Because $q = 2^Am-1,$ we know that
$2^A \mid   (q+1)$ and $q^2 - 1 = (q+1)(q-1)$ imply that if $u\in U_A,$ then
$u^{q+1} = 1$ and so $u\in \mathbb{F}_{q^2}.$ Note that since $q \equiv 3 \pmod{4},$ then
$q^2 \equiv 1 \pmod{4}.$ Then by Theorem \ref{2^nr and q = 1 mod 4}, Part 2 (a),
the complete factorization of $\Phi_{2^nr}$ over $\mathbb{F}_{q^2}$ is given by 
$$\Phi_{2^nr}(x) = \prod_{u\in U_A}\prod_{w\in
\Omega(r)}\left(\sum_{i=0}^{d_r}u^i S_{i,w}x^{2^{n-A}(d_r-i)}\right). \tab
\tab (**)$$ 
Let $Z_u(x) = \sum_{i=0}^{d_r}u^i S_{i,w}x^{2^{n-A}(d_r-i)}$ above, and since $u^q = u^{-1},$ 
consider its conjugate
$$
\overline{Z}_u(x) = \sum_{j=0}^{d_r}u^{-j} S_{j,w}x^{2^{n-A}(d_r-j)}.
$$

First, note that $u^{-1} \in U_A$ and $(**)$ imply $\overline{Z}_u$ is an
irreducible factor of $\Phi_{2^nr}$ over $\mathbb{F}_{q^2}.$ Moreover, $Z_u
\neq \overline{Z}_u.$ Indeed, ovserve that $u^{d_r} \neq u^{-d_r},$ otherwise
$u^{2d_r} = 1,$ and so $\ord(u) = 2^A$ gives $2^A \mid 2d_r$ contrary to $A
\geq 2$ and $d_r$ odd. Then $u^{d_r}S_{d_r,w} \neq u^{-d_r}S_{d_r,w}.$ As these
are the coefficients of $x^0$ in $Z_u(x),\ \overline{Z}_u(x),$ respectively,
necessarily $Z_u \neq \overline{Z}_u.$


We have 
$$
Z_u(x)\overline{Z}_u(x) =
\sum_{k=0}^{2d_r}\sum_{i+j=k}u^{i-j}S_{i,w}S_{j,w}x^{2^{n-A}(2d_r-k)}.
$$
Note from Part 2 and $(*)$ above that for
$u\in U_A$ we have $\sum_{i+j=k}u^{i-j}S_{i,w}S_{j,w} \in \F$ (since the
composed products of polynomials over $\F$ are polynomials over $\F$). Thus
$Z_u\overline{Z}_u \in \F[x],$ it has degree $2^{n-A+1}d_r,$
and is a factor of $\Phi_{2^nr}$ clearly. But then $Z_u\overline{Z}_u$ must be irreducible over $\F;$ hence the
complete factorization of $\Phi_{2^nr}$ over $\F$ must be 
$$\Phi_{2^nr}(x) =
\prod_{u\in U_A}\prod_{w\in \Omega(r)}\left(
\sum_{k=0}^{2d_r}\sum_{i+j=k}u^{i-j}S_{i,w}S_{j,w}x^{2^{n-A}(2d_r-k)} \right)$$
as required. 
\\
4. (iii) Similar to the proof of (i) in Theorem \ref{2^nr and q = 1 mod 4}.

(iv) Similar to the proof of (ii) in Theorem \ref{2^nr and q = 1 mod 4}.
\end{proof}

\begin{rem}
See Remark \ref{rem 2.1} after Theorem \ref{2^nr and q = 1 mod 4}. Furthermore,
comparing the factorizations in Parts 2 (i) and 3, we see that the factors in Part 3 can be obtained from the factors in Part 2 (i)
by the substitution $x \rightarrow x^{2^{n-A}}.$ Thus, for $n > A = v_2(q+1),$
if $\Phi_{2^Ar} = \prod_k h_{A_k}$ is the corresponding factorization,
then $\Phi_{2^nr}(x) = \prod_k h_{A_k}(x^{2^{n-A}})$ is the complete
factorization over $\F.$ Moreover, it is easy to see that $A = v_2(q+1)$ is the smallest such bound
with this property.
\end{rem}

\begin{rem}
In the case $d_r$ is even, see Remarks \ref{rem 2.2} and \ref{rem 2.3} after
Theorem \ref{2^nr and q = 1 mod 4}.
\end{rem}

\begin{rem}
Let $n > A,$ let $S = \{s_k\},\ T = \{t_k\}$ be homogeneous LRS's with
characteristic polynomials $\Phi_{2^n},\ \Phi_{r}$ respectively. Then as discussed earlier, 
the characteristic polynomial of $ST = \{s_k t_k \}$ is
$\Phi_{2^nr} = \Phi_{2^n} \odot \Phi_r.$ Suppose $d_r$ is odd. Since all
irreducible factors of $\Phi_{2^nr},\ n > A,$ have degree $2^{n-A+1}d_r,$ the minimal polynomial
of $ST$ must have degree $2^{n-A+1}d_r.$ This is the linear complexity of $ST.$
Note that if we let $n \rightarrow \infty,$ the linear complexity of the
corresponding LRS $ST$ approaches infinity.
\end{rem}

For the subcases $q \equiv 3 \pmod{4}$ with $q \equiv \pm 1 \pmod{r},$ and thus $d_r
= 1,\ 2,$ where $r$ is an odd prime, Theorem \ref{2^nr and q = 3 mod 4} becomes Theorem 1, Parts 1 and
3 in Fitzgerald and Yucas (2007) \cite{Fitzgerald 2007}.

\section{Conclusion}
In this paper we gave the factorization of the
cyclotomic polynomial $\Phi_{2^nr}$ over $\F$ where both $r \geq 3, \ q$
are odd and $\gcd(q,r) = 1.$ Previously, only $\Phi_{2^n3}$ and
$\Phi_{2^n5}$ had been factored in \cite{Fitzgerald 2007} and \cite{Prof},
respectively. As a result we have obtained infinite families of irreducible
sparse polynomials from these factors. Furthermore, we showed how to obtain the
factorization of $\Phi_n$ in a special case (see Theorem
\ref{cyclotomics are composed}). We also showed in Theorem \ref{cyclo and
minimal} how to obtain the factorization of $\Phi_{mn}$ from the
factorization of $\Phi_n$ when $q$ is a primitive root modulo $m$ and $\gcd(m,n) =
\gcd(\phi(m),\ord_n(q)) = 1.$
 
The factorization of $\Phi_{2^n}$ was already given in \cite{Lidl} when $q
\equiv 1 \pmod{4}$ and in \cite{Meyn} when $q \equiv 3 \pmod{4}.$ It is 
natural to consider the factorization of $\Phi_{3^n}.$ We then wonder
if some of the techniques used in Section 3 could be applied to factor
$\Phi_{3^nr} = \Phi_{3^n}\odot \Phi_r.$ In particular, it would be
desirable to generalize Theorem \ref{thm 5} to allow for other cases
(besides $2^n$). It is expected that these irreducible factors will be sparse as
well. Note that we can allow $q$ to be even in this case by forcing $r$ to be odd. 
This is significant as the fields $\mathbb{F}_{2^m}$ are the most commonly used in modern engineering.

In Section 2 we considered irreducible composed products of the form $f \odot
\Phi_m.$ In particular, we derived the construction of a new class of
irreducible polynomials in Theorem \ref{thm 3}. It is natural to consider other
classes of polynomials and substitute them for $\Phi_m$ and see what the
result may be. 

We also gave formulas for the linear complexity of $ST$ when $\Phi_{2^n},\
\Phi_r$ are characteristic polynomials of the homogeneous LRS's $S,\ T,$
respectively. We showed that by letting $n \rightarrow \infty,$ the linear complexity of $ST$
will approach infinity.

Another matter of interest is the factorization of composed products. Since the
minimal polynomial of a LRS, say $ST,$ is an irreducible factor of some composed
product, this has applications in stream cipher theory, LFSR and LRS in general. 
D. Mills (2001) \cite{Mills} had already studied the factorization of arbitrary
composed products. In particular, if $\deg f = m$ and $\deg g = n$ with $f,\ g$ irreducible over $\F,$ Mills gave $d = \gcd(m,n)$ as an upper bound for
the number of irreducible factors that $f \diamond g$ could decompose into. He
also gave the possible degrees that these irreducible factors may attain. As a result, we
now know the possible linear complexities that $ST$ could attain. On the other
hand his work was generalized for two arbitrary irreducible polynomials $f$ and
$g.$ In the case that at least one of these polynomials belongs to a certain
class of polynomials with well defined properties, we wonder if it could be
possible to obtain more precise information regarding the number of
irreducible factors and their degrees. For instance,
in the case of $f \odot \Phi_m,$ can we know precisely the degrees of the irreducible factors? Can we
know precisely in how many irreducible factors does $f \odot \Phi_m$ decompose
into? Note that the subject of the factorization of composed products is one for which
very little research has been done. Currently, the authors were able to find
only one paper \cite{Mills} on this matter and they feel this is a topic that
has been somewhat neglected.

\appendix

\section{Samples of Irreducible Polynomials $F_m$}

We provide a table of examples for Theorem \ref{thm 3}. MAPLE software was
used in the computations.\\

Table 1. Table of (irreducible) samples of $F_m$ from Theorem \ref{thm 3}
outputed on inputs $(m,q,n)$ and $f.$
  
  \begin{tabular}{| c | p{5cm} | p{9cm} | }
    
    \hline
    $\left(m,q,n\right)$ & $f(x)$ & $F_m(x)$ \\ \hline
    $(2,3,6)$ & $x^6+ 2x^4 + x^3 +2x+1 $ & $x^6 + x^5+2x^4+x^3+x+2$ \\ \hline
    $(2,5,5)$ & $x^5+3x^4+4x^3+4x+2$ & $x^5+2x^4+4x^3+4x+3$ \\	\hline
    $(4,3,9)$ &
    $x^9+x^7+x^6+x+1$&$x^{18}+x^{16}+x^{14}+x^{12}+2x^{10}+x^8+x^6+x^2+1$\\\hline
    $(4,7,3)$ & $x^3+4x^2+1$ & $x^6+2x^4+6x^2+1$\\ \hline
    $\left(3^2,5,5\right)$ & $x^5+3x^4+4x^2+x+1$ &
    $x^{30}+3x^{27}+3x^{24}+3x^{21}+3x^{18}+x^{15}+2x^9+4x^6 + 2x^3+1$\\ \hline
    $\left(7^2,3,5\right)$ & $x^5+x^4+x^2+2x+2$ & $x^{210} + 2x^{203}+ \dots +
    1$ \\ \hline
    $(6,5,9)$ & $x^9+4x^8+3x^7+x^5+3x^4+4x^2+2x+3$
    &$x^{18}+4x^{17}+3x^{16}+2x^{15}+3x^{14}+x^{11}+x^{10}+2x^9+4x^8+x^7+x^6+x^5+2x^4+3x^3+2x^2+x+4$\\
    \hline
    $(10,3,5)$ & $x^5+x^3+x^2+2x+2$ &
    $x^{20}+2x^{18}+x^{17}+2x^{16}+x^{15}+x^{14}+x^{12}+2x^{10}+2x^8+x^7+2x^3+2x^2+x+1$\\
    \hline
    $\left(3^2,2,5\right)$ & $x^5+x^2+1$ & $x^{30} + x^{27}+ x^{21}+ x^6+1$\\
    \hline
    $\left(3^3,2,5\right)$ & $x^5+x^2+1$ &
    $x^{90}+x^{81}+x^{72}+x^{45}+x^{27}+x^9+1$\\ \hline
  \end{tabular}

\section{Recursive Computations}

We provide the following tables of examples for Theorems $\ref{2^nr and q = 1
mod 4}$ (i) and $\ref{2^nr and q = 3 mod 4}$ (iii).
The coefficients $(a_{n_1},a_{n_2},\dots , a_{n_6})$ are the
coefficients of the irreducible factors of $\Phi_{2^nr}$ over $\F$
for $q = 5,\ 19,\ r = 7,\ n\leq K = 3,$ calculated by using the recurrence
relations in Theorems $\ref{2^nr and q = 1 mod 4}$ (i) and $\ref{2^nr and q = 3
mod 4}$ (iii). In particular, the tables show that these recursive relations,
now with initial values corresponding to $n = 1,$ may be used to obtain the
factors of $\Phi_{2^nr}$ when $n \leq A$ as well. MAPLE
software was used in the computations.\\

Table 2. Factorization of $\Phi_{2^nr}$ over $\F$ where $r = 7,\ q = 5,\ n
\leq K= 3$

\begin{tabular}{| c | c | c | c |}
\hline
$n$ & 1 & 2 & 3\\ \hline
$(a_{n_1},a_{n_2},\dots , a_{n_6})$ & (4, 1, 4, 1, 4, 1) & (2, 4, 3, 1, 2, 4)
&(1, 4, 3, 2, 4, 2)\\ 
&& (3, 4, 2, 1, 3, 4) & (4, 4, 2, 2, 1, 2)\\
&& & (2, 1, 4, 2, 3, 3)\\
&& & (3, 1, 1, 2, 2, 3)\\
\hline
\end{tabular}
\\ \\

Table 3. Factorization of $\Phi_{2^nr}$ over $\F$ where $r = 7,\ q = 19,\ n
\leq K = 3$

\begin{tabular}{|c | c |c | c |}
\hline
$n$ & 1 & 2 & 3\\ \hline
$(a_{n_1},a_{n_2},\dots , a_{n_6})$ & (18, 1, 18, 1, 18, 1) & (8, 3, 8, 3, 8, 1)
& (2, 6, 10, 13, 2, 18)\\
&& (11, 3, 11, 3, 11, 1) & (17, 6, 9, 13, 17, 18)\\
&& & (8, 9, 18, 10, 8, 18)\\
&& & (11, 9, 1, 10, 11, 18)\\
\hline
\end{tabular}

\end{document}